\documentclass[12pt,reqno]{amsart}

\usepackage[linktoc=page,colorlinks=true,linkcolor=blue,citecolor=blue]{hyperref}
\usepackage{array}
\usepackage{amssymb,amsmath,amsthm,latexsym}
\usepackage{bm,bbding}
\usepackage{cleveref}
\usepackage{enumerate}
\usepackage{tikz-cd}
\usetikzlibrary{arrows.meta}                   
\tikzcdset{arrow style=tikz, diagrams={>=stealth}}
\usepackage{spectralsequences} 
\usepackage{pstricks-add}
\usepackage[all]{xy}

\makeatletter
\@namedef{subjclassname@2020}{%
  \textup{2020} Mathematics Subject Classification}
\makeatother

\newtheorem{theorem}{Theorem}[section]
\newtheorem{lemma}[theorem]{Lemma}
\newtheorem{proposition}[theorem]{Proposition}
\newtheorem{corollary}[theorem]{Corollary}
\newtheorem{definition}[theorem]{Definition}

\numberwithin{equation}{section}
\theoremstyle{definition}
\newtheorem{remark}[theorem]{Remark}

\newtheorem{example}[theorem]{Example}

\newcommand{\rp}{\mathbb{R}P}
\newcommand{\cp}{\mathbb{C}P}
\newcommand{\hp}{\mathbb{H}P}
\newcommand{\fp}{\mathbb{F}P}
\newcommand{\ff}{\mathbb{F}}
\newcommand{\rr}{\mathbb{R}}
\newcommand{\cc}{\mathbb{C}}
\newcommand{\hh}{\mathbb{H}}
\newcommand{\qq}{\mathbb{Q}}
\newcommand{\zz}{\mathbb{Z}}

\allowdisplaybreaks

\begin{document}

\title[Cohomology algebra of orbit space]{Cohomology algebra of orbit spaces of free involutions on the product of projective space and $4$-sphere}
\author{Ying Sun, Jianbo Wang}
\address{(Ying Sun) School of Mathematics, Tianjin University, Tianjin 300350, China}
\email{sy0097265@163.com}
\address{(Jianbo Wang) School of Mathematics, Tianjin University, Tianjin 300350, China}
\email{wjianbo@tju.edu.cn}

\thanks{This work was supported by the Natural Science Foundation of Tianjin City of China (Grant No. 19JCYBJC30300) and the National Natural Science Foundation of China (Grant No. 12071337).} 
\begin{abstract}
Let $X$ be a finitistic space with the mod $2$ cohomology of the product space of a projective space and a $4$-sphere. Assume that $X$ admits a free involution. 
In this paper we study the mod $2$ cohomology algebra of the quotient of $X$ by the action of the free involution  and derive some consequences regarding the existence of $\zz_2$-equivariant maps between such $X$ and an $n$-sphere.
\end{abstract}
\date{}
\maketitle
\tableofcontents

\section{Introduction}
The study of the orbit space of a topological group $G$-action on a topological space $X$ is a classical topic in topology. In particular, the finitistic space plays an important role in the cohomology theory of transformation groups. A paracompact Hausdorff space $X$ is said to be \emph{finitistic} if every open covering of $X$ has a finite dimensional open refinement, where the dimension of a covering is one less than the maximum number of members of the covering which intersect nontrivially. Finitistic spaces behave nicely under compact Lie group $G$ actions. More precisely, the space $X$ is finitistic if and only if the orbit space $X/G$ is finitistic (\cite{DeoTripathi1982,DeoS.Singh1982}). 

For a given topological space $X$ with the action of a topological group $G$, it is often difficult to determine the topological type or homotopy type of $X/G$. Orbit spaces of free actions of finite groups on spheres have been studied extensively by Livesay \cite{Livesay1960}, Rice \cite{Rice1969}, Ritter \cite{Ritter1973}, Rubinstein \cite{Rubinstein1979} and many others. Tao \cite{Tao1962} determined orbit spaces of free involutions on $S^1\times S^2$. Later Ritter \cite{Ritter1974} extended the results to free actions of cyclic groups of order $2^n$. However, there are few known results on compact manifolds other than a sphere. Hence we try to determine the cohomology algebra of the orbit space of some more examples. 

To deal with more general spaces, by the notation $X\sim_{\qq} Y$  (resp. $X\sim_p Y$, $p$  a prime), we mean that $X$ and $Y$ have the same rational (resp. mod $p$) cohomology algebras, not necessarily induced by a map between $X$ and $Y$. Let’s list some related results.
\begin{itemize}
\item R.M. Dotzel and others (\cite{Dotzel2001}) have determined the cohomology algebra of  orbit spaces of $\zz_p$-action (resp. $S^1$-action)  on a finitistic space $X\sim_p S^m\times S^n$ (resp. $X\sim_\qq S^m\times S^n$). 
\item H.K. Singh and T.B. Singh have determined  the mod $2$ cohomology algebras of orbit spaces of  free $\zz_2$-action on a finitistic space $X\sim_2\rp^n$ and $X\sim_2\cp^n$ in \cite{HKSinghTEJBSingh2008}, and also determined the mod $p$ and rational (resp. mod $p$) cohomology algebras of orbit spaces of free $S^1$-action on a finitistic space $X\sim_F S^1\times  \cp^{m-1}$ with $F=\zz_p$ or $\qq$ (resp. mod $p$ cohomology lens space $X\sim_p L^{2m-1}(p;q_1,\dots, q_m)$) in \cite{HKSinghTBSingh2010}.  
\item M. Singh has determined the cohomology algebras of orbit spaces of free involutions on  a finitistic space $X\sim_2 \rp^n\times \rp^m$,  $X\sim_2 \cp^n\times \cp^m$ in \cite{MSingh2010} and $X\sim_2 L^{2m-1}(p;q_1,\dots, q_m)$ in \cite{MSingh2013}. 
\item P. Dey and M. Singh have calculated the mod $2$ cohomology algebras of orbit spaces of free $\zz_2$ and $S^1$-action on a compact Hausdorff space with mod $2$ cohomology algebra of a real or complex Milnor manifold (\cite{DeyMSingh2018}).
\item A.M.M. Morita et al have calculated the possible $\zz_2$-cohomology rings of orbit spaces of free actions of $\zz_2$ (or fixed point free involutions) on the Dold manifold $P(1,n)$ with $n$ odd (\cite{MoritaMattosPergher2018}).
\item P. Dey has determined the possible mod 2 cohomology algebra of orbit spaces of free involutions on a finite dimensional CW-complex homotopic to Dold manifold $P(m,n)$ (\cite{PDey2018}).
\item In \cite{SHTSingh2017}, S.K. Singh and others have determined the cohomology algebra of orbit spaces of free involutions on a finitistic space $X\sim_2 \fp^m\times S^3$, where $\fp^m$ is a projective space, and $\ff$ stands for either the field $\rr$ of real numbers, the field $\cc$ of complex numbers or the division ring $\hh$ of quaternions. 

\item As applications of cohomology algebras, the existence of $\zz_2$-equivariant maps $X\to S^n$ or $S^n\to X$ is discussed in \cite{DeyMSingh2018, HKSinghTEJBSingh2008, MSingh2010, MSingh2013, SHTSingh2017}.   
\end{itemize}

This paper deals with the free action of $\zz_2$ on a finitistic space $X$ with mod $2$ cohomology of the product of a projective space and $4$-sphere, i.e. a space $X\sim_2 \fp^m \times S^4$, along with the cohomology algebra of orbit spaces under free involutions. 

The paper is organized as follows. In \Cref{sec-Pre}, we recall the Leray-Serre spectral sequence associated to the Borel fibration $X\hookrightarrow X_G \to B_G$, and list some known results. \Cref{sec-mainresults} consists of three main Theorems \ref{thm-Z2faRPmS4}, \ref{thm-Z2faCPmS4}, \ref{thm-Z2faHPmS4} and two Lemmas \ref{lem-acttri}, \ref{lem-acttri-Z}.  In \Cref{sec-proofs}, we prove three main theorems which describe the possible cohomology algebras of orbit spaces. In the last \Cref{sec-equivmap}, as applications of the main theorems, we discuss the existence of $\zz_2$-equivariant maps $X\to S^n$ or $S^n\to X$.

\section{Preliminaries}\label{sec-Pre}

We now recall the Borel construction and some results on its spectral sequence. Let $G$ be a compact Lie group acting on a finitistic space $X$. Let 
\(
E_G \to B_G
\) 
be the universal principal $G$-bundle. The \emph{Borel construction} on $X$ is defined as the orbit space 
\[
X_G=(X\times E_G)/G,
\] 
where $G$ acts diagonally (and freely) on the product $X\times E_G$. The projection $X\times E_G\to E_G$ gives a fibration (\cite[Chapter IV]{Borel1960}), called the \emph{Borel fibration},
\[
X\overset{i}{\hookrightarrow} X_G \xrightarrow[]{\pi} B_G.
\] 
Throughout, we use the \v Cech cohomology with $\mathbb{Z}_2$ coefficients, and suppress it from the notation. 

We exploit the Leray-Serre spectral sequence $\{E_r^{k,l},d_r\}$ associated to the Borel fibration $X\overset{i}{\hookrightarrow} X_G \xrightarrow[]{\pi } B_G$ (\cite[Theorem 5.2]{Mccleary2001}), such that: 
\begin{itemize}
\item[(1)]
$d_r:E_r^{k,l}\to E_r^{k+r,l-r+1}$, and 
\[
E_{r+1}^{k,l}=\dfrac{\ker d_r:E_r^{k,l}\to E_r^{k+r,l-r+1}}{\text{im }d_r:E_r^{k-r,l+r-1}\to E_r^{k,l}}.
\]
\item[(2)] The infinity terms $E_{\infty}^{k,n-k}$ is isomorphic to the successive quotients $F_k^n/F_{k+1}^n
$ in a filtration  $0\subset F_n^n\subset\cdots \subset F_1^n \subset F_0^n=H^n(X_G)$ of $H^n(X_G)$.
\item[(3)] The $E_2$-term of this spectral sequence is given by 
\[
E_2^{k,l}=H^k(B_G;\mathcal{H}^l(X)), 
\] 
where $\mathcal{H}^l(X)$ is a locally constant sheaf with stalk $H^l(X)$, and the $E_2$-term converges to $H^*(X_G)$ as an algebra. 
\end{itemize}
If $\pi _1(B_G)$ acts trivially on $H^*(X)$, then the system of local coefficients is simple, that is, the cohomology with local coefficients $H^k(B_G;\mathcal{H}^l(X))$ is just the (ordinary) cohomology $H^k(B_G;H^l(X))$ so that, by the universal coefficient theorem, we have 
\[
E_2^{k,l}\cong H^k(B_G)\otimes H^l(X).
\]
Further, if the system of local coefficients is simple, the restriction of the product structure in the spectral sequence to the subalgebras $E_2^{*,0}$ and $E_2^{0,*}$ coincide with the cup products on $H^*(B_G)$ and $H^*(X)$, respectively. The edge homomorphisms 
\begin{align*}
& H^k(B_G)\cong E_2^{k,0}\twoheadrightarrow E_3^{k,0} \twoheadrightarrow \dots \twoheadrightarrow E_k^{k,0}\twoheadrightarrow E_{k+1}^{k,0}=E_\infty ^{k,0}\subset H^k(X_G)\text{~and}\\ 
& H^l(X_G)\twoheadrightarrow E_\infty ^{0,l} =E_{l+2}^{0,l}\subset E_{l+1}^{0,l}\subset \dots \subset E_2^{0,l} \cong H^l(X)
\end{align*}
are the homomorphisms
\begin{align*}
& \pi ^*:H^k(B_G)\to H^k(X_G) \text{~and~} \\
& i^*:H^l(X_G) \to H^l(X)
\end{align*}
respectively. 
The graded commutative algebra $H^*(X_G)$ is isomorphic to $\text{Tot}E_\infty ^{*,*}$, the total complex of $E_\infty ^{*,*}$, given by
\[
(\text{Tot}E_\infty ^{*,*})^{q}=\bigoplus_{k+l=q}E_{\infty}^{k,l}.
\]

Next, we recall some known results.
\begin{proposition}[{\cite[Corollary 9.6]{Loring2020}}]\label{Loringcor96} 
If a topological group $G=\mathbb{Z}_2$ acts freely on a topological space $X$ such that $X\to X/G$ is a principal $G$-bundle, then the equivariant cohomology $H^*_G(X)=H^*(X_G)$ is isomorphic to $H^*(X/G)$.
\end{proposition}
\begin{proposition}[{\cite[Theorem 1.5, p.374]{Bredon1972}}]\label{Bredonthm15}
Let $G=\mathbb{Z}_2$ act on a finitistic space $X$ with $H^i(X)=0$ for all $i>n$. Then $H^i(X_G)$ is isomorphic to $H^i(X^G)$ for $i>n$, where $X^G$ is the fixed point set of the $G$-action.
\end{proposition}
\begin{proposition}[{\cite[Corollary 7.2, p.406]{Bredon1972}}]\label{Bredoncor72}
Let $G=\mathbb{Z}_2=\langle g\rangle$ act on a finitistic space $X$. Then the element $cg^*(c)\in H^{2n}(X)^G=H^0(B_G;H^{2n}(X))=E_2^{0,2n}$ is a permanent cocycle in the spectral sequence of $X\hookrightarrow X_G \to B_G$, for any $c\in H^n(X)$. 
\end{proposition}
\begin{proposition}[{\cite[Theorem 7.4, p.407]{Bredon1972}}]\label{Bredonthm74}
Let $G=\mathbb{Z}_2=\langle g\rangle $ act on a finitistic space $X$. Suppose that $H^i(X)=0$ for all $i>2n$ and $H^{2n}(X)=\mathbb{Z}_2$. Suppose that $c\in H^n(X)$ is an element such that $cg^*(c)\ne 0$, then the fixed point set is non-empty.
\end{proposition}
\begin{proposition}[{\cite[Corollary 7.5, p.407]{Bredon1972}}]\label{Bredoncor75}
Let $G=\mathbb{Z}_2=\langle g\rangle $ act on a finitistic space $X\sim_2 S^n\times S^n$ and suppose that $g^*\ne 1$ on $H^n(X)$. Then the fixed point set is non-empty.
\end{proposition}  

\section{Cohomology algebra of orbit space of free $\zz_2$-action on $X\sim_2\mathbb{F}P^m\times S^4$}\label{sec-mainresults} 
  
Assume that $X$ is a finitistic space equipped with a free involution and has the mod $2$ cohomology of $\mathbb{F}P^m\times S^n$, i.e.,
\[
H^*(X)=\mathbb{Z}_2[a,b]/\langle a^{m+1},b^2\rangle,
\]
where,  $\deg a=\lambda$,  when $\mathbb{F}=\mathbb{R}, \mathbb{C}$ or $\mathbb{H}$, $\lambda=1$, $2$ or $4$, respectively, and $\deg b=n$. 
Now, we present three main theorems of this paper. More concretely, we determine the cohomology algebras of orbit spaces of free involutions on $X\sim _2 \mathbb{F}P^m \times S^4$.

\begin{theorem}\label{thm-Z2faRPmS4}
Let $G=\zz_2$ act freely on a finitistic space $X\sim _2 \mathbb{R}P^m \times S^4$. If $m=5$ or $m=7$, assume further that the action of $G$ on $H^*(X;\zz_2)$ is trivial or $X\sim_{\zz} \mathbb{R}P^m \times S^4$. Then $H^*(X/G)$ is isomorphic to one of the following graded commutative algebras 
\begin{alignat*}{2}
&\zz_2[x,y,z]/I_1, && ~\deg x=1, \deg y=2, \deg z=4;\\
&\zz_2[x,y]/I_k, &&  ~\deg x=1, \deg y=1, ~k=2,3,\dots,9.
\end{alignat*}
Where the ideal $I_k$ is listed as follows:

$(\mathrm{1})$ $I_1=\langle x^2,y^{\frac{m+1}{2}},z^2\rangle$, where $m$ is odd.

$(\mathrm{2})$ $I_2=\langle x^5,y^{m+1}+\alpha_1xy^m+\alpha _2x^2y^{m-1}+\alpha _3x^3y^{m-2}+\alpha _4x^4y^{m-3}\rangle$, where $\alpha _i\in \zz_2$, $ i=1,\dots ,4$. If $m=1$, then $\alpha_3=\alpha_4=0$. If $m=2$, then $\alpha_4=0$.

$(\mathrm{3})$ $I_3=\langle x^{m+5},y^{m+1}+\alpha _1xy^m+\alpha _2x^2y^{m-1}+\alpha _3x^3y^{m-2}+\alpha _4x^{m+1},x^4y\rangle$, where $\alpha _i \in \zz_2$, $ i=1,\dots ,4$. If $m=1$, then $\alpha_3=\alpha_4=0$. If $m=2$, then $\alpha_4=0$.

$(\mathrm{4})$ $I_4=\langle x^{m+5},y^{m+1}+\alpha _1xy^m+\alpha _2x^2y^{m-1}+\alpha _3x^my+\alpha _4x^{m+1},x^3y^2+\beta _1x^4y+\beta_2 x^5,x^{m+3}y\rangle$, where $m\geqslant 2$ and $\alpha _i, \beta_1, \beta_2 \in \zz_2$, $ i=1,\dots ,4$. If $m=2$, then $\alpha_3=0$.

$(\mathrm{5})$ $I_5=\langle x^{m+4},y^{m+1}+\alpha _1xy^m+\alpha _2x^2y^{m-1}+\alpha _3x^my+\alpha _4x^{m+1},x^3y^2+\beta _1x^4y+\beta_2 x^5 \rangle$, where $m\geqslant 2$ and $\alpha _i, \beta_1, \beta_2\in \zz_2$, $ i=1,\dots ,4$. If $m=2$, then $\alpha_3=0$.

$(\mathrm{6})$ $I_6=\langle x^{m+5},y^{m+1}+\alpha _1xy^m+\alpha _2x^{m-1}y^2+\alpha _3x^my+\alpha _4x^{m+1}, x^2y^3+\beta _1x^3y^2+\beta _2x^4y+\beta_3 x^5, x^{m+1}y^2+\gamma _1x^{m+2}y+\gamma _2x^{m+3}, x^{m+3}y\rangle$, where  $m\geqslant 3$ and $\alpha_i, \beta_j, \gamma_1, \gamma_2 \in \zz_2$, $ i=1,\dots ,4$, $j=1,2,3$.

$(\mathrm{7})$ $I_7=\langle x^{m+4},y^{m+1}+\alpha _1xy^m+\alpha _2x^{m-1}y^2+\alpha _3x^my+\alpha _4x^{m+1}, x^2y^3+\beta _1x^3y^2+\beta _2x^4y+\beta_3 x^5, x^{m+1}y^2+\gamma _1x^{m+2}y+\gamma _2x^{m+3}\rangle$, where  $m\geqslant 3$ and $\alpha _i, \beta_j, \gamma_1, \gamma_2\in \zz_2$, $ i=1,\dots ,4$, $j=1,2,3$.

$(\mathrm{8})$ $I_8=\langle x^{m+5}, y^{m+1}+\alpha _1xy^m+\alpha _2x^{m-1}y^2+\alpha _3x^my+\alpha _4x^{m+1},x^2y^3+\beta _1x^3y^2+\beta _2x^4y+\beta_3 x^5, x^{m+2}y\rangle$, where  $m\geqslant 3$ and $\alpha _i, \beta_j \in \zz_2$, $ i=1,\dots ,4$, $j=1,2,3$.

$(\mathrm{9})$ $I_9=\langle x^{m+3},y^{m+1}+\alpha _1xy^m+\alpha _2x^{m-1}y^2+\alpha _3x^my+\alpha _4x^{m+1},x^2y^3+\beta _1x^3y^2+\beta _2x^4y+\beta_3 x^5\rangle$, where $m\geqslant 3$ and $\alpha _i, \beta_j\in \zz_2$, $ i=1,\dots ,4$, $j=1,2,3$.
\end{theorem}

\begin{theorem}\label{thm-Z2faCPmS4}
Let $G=\zz_2$ act freely on a finitistic space $X\sim _2 \mathbb{C}P^m \times S^4$. If $m=3$, assume further that the action of $G$ on $H^*(X;\zz_2)$ is trivial or $X\sim _{\zz} \mathbb{C}P^3 \times S^4$. Then $H^*(X/G)$ is isomorphic to one of the following graded commutative algebras 
\begin{alignat*}{2}
&\zz_2[x,y,z]/I_1, &&~ \deg x=1, \deg y=4, \deg z=4;\\
&\zz_2[x,y]/I_k, &&~ \deg x=1, \deg y=2, k=2,3. 
\end{alignat*}
Where the ideal $I_k$ is listed as follows:

$(\mathrm{1})$ $I_1=\langle x^3,y^{\frac{m+1}{2}},z^2\rangle$, where $m$ is odd.

$(\mathrm{2})$ $I_2=\langle x^5,y^{m+1}+\alpha_1x^2y^m+\alpha _2x^4y^{m-1}\rangle$, where $\alpha _1, \alpha_2\in \zz_2$.

$(\mathrm{3})$ $I_3=\langle x^{2m+5},y^{m+1}+\alpha _1x^2y^m+\alpha _2x^{2m+2},x^3y\rangle$, where $\alpha_1, \alpha_2 \in \zz_2$.
\end{theorem}
\begin{theorem}\label{thm-Z2faHPmS4}
Let $G=\zz_2$ act freely on a finitistic space $X\sim _2 \mathbb{H}P^m \times S^4$. When $m\equiv 3\pmod 4$, assume further that the action of $G$ on $H^*(X;\zz_2)$ is trivial or $X\sim _{\zz} \mathbb{H}P^m \times S^4$. Then $H^*(X/G)$ is isomorphic to one of the following graded commutative algebras 
\begin{alignat*}{2}
&\zz_2[x,y,z]/I_1, &&~\deg x=1, \deg y=8, \deg z=4;\\
& \zz_2[x,y]/I_2, &&~\deg x=1, \deg y=4.
\end{alignat*}
Where the ideal $I_1$ and $I_2$ are as follows:

$(\mathrm{1})$ $I_1=\langle x^5,y^{\frac{m+1}{2}}+\beta x^4y^{\frac {m-1}{2}}z, z^2+\gamma y+\alpha x^4z\rangle$, where $\alpha,\beta,\gamma\in \zz_2$ and $m$ is odd. If $m=1$, then $\beta=\gamma=0$. 

$(\mathrm{2})$ $I_2=\langle x^5,y^{m+1}\rangle$.
\end{theorem}

\begin{example}	
When $m$ is odd there are standard free involutions of $\rp^m$ and $\cp^m$. The map
\[
[x_{0},x_{1}, \dots, x_{m-1}, x_{m}]\mapsto [-x_{1},x_{0}, \dots,  -x_{m}, x_{m-1}]
\]  
defines a free involution of $\rp^m$ with the orbit space $\rp^m/\zz_2\sim_2 S^1\times \cp^{\frac{m-1}{2}}$ (\cite[Example 3.3]{SHTSingh2017}). Quotienting by the product of the above map with the trivial $\zz_2$-action on $S^4$, the mod $2$ cohomology algebra of the orbit space $\rp^m\times S^4/\zz_2$ is that of $S^1\times \cp^{\frac{m-1}{2}}\times S^4$, which account for case (1) in Theorem \ref{thm-Z2faRPmS4}.

Similarly, the map 
\[
[z_{0}:z_{1}: \dots : z_{m-1}: z_{m}]\mapsto [-\overline{z_{1}}:\overline{z_{0}}:  \dots:  -\overline{z_{m}}:\overline{z_{m-1}}]
\]
defines a free quaternionic involution of $\cp^m$ with  the orbit space  $\cp^m/\zz_2\sim_2\rp^2\times \hp^{\frac{m-1}{2}}$ (\cite[Example 3.7]{SHTSingh2017}). Quotienting by the product of the above map with the trivial $\zz_2$-action on $S^4$, the mod $2$ cohomology algebra of the orbit space $\cp^m\times S^4/\zz_2$ is that of $\rp^2\times \hp^{\frac{m-1}{2}}\times S^4$, which account for case (1) in Theorem \ref{thm-Z2faCPmS4}. 

The same construction as above does not apply to $\hp^m\times S^4$. For $m=1$, namely, $X\sim _2 S^4\times S^4$, by \Cref{Bredoncor75} we see that there is no free involution on $X$. For $m>1$, $\hp^m$ has the fixed point property (\cite[Example 4L.4]{Hatcher2002}), where the fixed point property of a topological space means that every continuous map (not necessarily a self-homeomorphism) from the topological space to itself has a fixed point. 
\end{example}

\begin{example}
Consider the trivial $\zz_2$-action on $\fp^m$ and the antipodal action of $\zz_2$ on $S^4$, then the orbit space of the free involution on $\fp^m\times S^4$ is $\fp^m\times \rp^4$. The cohomology algebra of $\fp^m\times \rp^4$ account for the cases (2) of main Theorems \ref{thm-Z2faRPmS4}, \ref{thm-Z2faCPmS4} and \ref{thm-Z2faHPmS4} with all coefficients zero.

When $\alpha_1=\alpha_2=0$, \Cref{thm-Z2faCPmS4} (2)  describes the cohomology ring of the Dold manifold $P(4,m)$. The Dold manifold $P(n,m)$ is the orbit space of $S^n\times \cp^m$ by the free involution
that acts antipodally on $S^n$ and by complex conjugation on $\cp^m$. Following \cite{Dold1956}, the ring structure of $H^*(P(n,m))$ is given by 
\[
H^*(P(n,m))=\zz_{2}[x,y]/\langle x^{n+1}, y^{m+1}\rangle, 
\]
where $\deg x=1, \deg y=2$.
\end{example}

An open question coming from Theorems \ref{thm-Z2faRPmS4}, \ref{thm-Z2faCPmS4} and \ref{thm-Z2faHPmS4} is to search for possible more exotic free involutions and identify the respective cohomology algebras. 

The proofs of the above three main theorems are based on spectral sequence arguments. To make the calculation of spectral sequence easier, we firstly prove the following  general result, which is an extension of \cite[Lemma 3.1]{SHTSingh2017}.

\begin{lemma}\label{lem-acttri} 
Let $G=\mathbb{Z}_2$ act freely on a finitistic space $X\sim _2 \mathbb{F}P^m \times S^n$, where $\mathbb{F}=\mathbb{R}, \mathbb{C}$ or $\mathbb{H}$. Let $\lambda=1$, $2$ or $4$, respectively. Then the action of $G$ on $H^*(X;\zz_2)$ is trivial with possibly two exceptions: 

{\rm (i)} $m\equiv 3\pmod 4$ and $n=\lambda$; 

{\rm (ii)} $\lambda m=n+j$, $j\equiv \lambda \pmod{2\lambda}$, $0\leqslant j<n$ and $\frac{n}{\lambda}\equiv 0 \pmod{2}$. 
\end{lemma}
\begin{proof}
The mod $2$ cohomology algebra $H^*(X;\zz_2)$ has two generators $a$ and $b$ satisfying $a^{m+1}=0$ and $b^2=0$. Let $g$ be the generator of $G=\zz_2$. By the naturality of the cup product, we get
\[
g^*(a^ib)=g^*(a)^ig^*(b) \text{~for all $i\geqslant 0$,}
\]
where $g^*$ is the mod $2$ cohomology isomorphism $H^*(X;\zz_2)\to H^*(X;\zz_2)$. 

Firstly, we claim that 
\begin{center}
$g^*(a)=a$, except the case: $m\equiv 3\pmod 4$ and $n=\lambda$.
\end{center}

$\blacktriangleright$ If $\deg a=\lambda\ne n=\deg b$, we clearly have $g^*(a)=a$. 

$\blacktriangleright$ For $m=1$, $n=\lambda$, the mod $2$ cohomology of $X$ is the same as of
\[
S^1\times S^1, ~S^2\times S^2 ~\text{or~} S^4\times S^4.
\]
If $G$ acts nontrivially on $H^*(X;\zz_2)$, by \Cref{Bredoncor75}, we have $X^{G}$ is non-empty, which contradicts the action being free.

$\blacktriangleright$ For $m>1$ and $m\equiv 1\pmod 4$, $n=\lambda$. Since the orders of $a$ and $b$ are $m+1$ and $2$, respectively, it follows that $g^*(a)\ne b$. Let $c=a^{\frac{m+1}{2}}\in H^{\lambda\frac{m+1}{2}}(X;\zz_2)$. If $g^*(a)=a+b$, then 
\[
cg^*(c)=a^{\frac{m+1}{2}}(a+b)^{\frac{m+1}{2}}=a^{\frac{m+1}{2}}(a^{\frac{m+1}{2}}+\tfrac{m+1}{2}a^{\frac{m-1}{2}}b)=a^mb\ne 0. 
\] 
By \Cref{Bredonthm74}, $X^G$ is non-empty, which contradicts the action being free. So $g^*(a)=a$.
 

$\blacktriangleright$ For $m>1$ even, $n=\lambda$. If $g^*(a)=a+b$, then $a^{m+1}=0$ gives $0=g^*(a^{m+1})=(a+b)^{m+1}=(m+1)a^mb=a^mb$, a contradiction. 

Therefore, except for the case when $m\equiv 3\pmod 4$ and $n=\lambda$, we have $g^*(a)=a$ and \Cref{lem-acttri} is reduced to show that
\[
g^*:H^n(X;\zz_2)\to H^n(X;\zz_2)
\]
is the identity isomorphism. 
	
	If \underline{$\lambda\nmid n$ or $\lambda m<n$}, the cohomology group $H^j(X;\zz_2)$ is $\zz_2$ or zero for any $j\geqslant 0$,  \Cref{lem-acttri} is obvious.  Thus we need to consider that \underline{$\lambda \mid n$ and $\lambda m\geqslant n$, $m>1$}. 
	
	If $G$ acts nontrivially on $H^*(X;\zz_2)$, then we get $g^*(b)=a^{\frac{n}{\lambda}}$ or $g^*(b)=a^{\frac{n}{\lambda}}+b$. If $g^*(b)=a^{\frac{n}{\lambda}}$, then $g^*(a^mb)=a^{m+\frac{n}{\lambda}}=0$. Since $g^*$ is an isomorphism, this gives $a^mb=0$, which is a contradiction. So we must have 
	\begin{equation}\label{eq-g*banlambda+b}
		g^*(b)=a^{\frac{n}{\lambda}}+b. 
	\end{equation}
	From now on to the end of the proof of \Cref{lem-acttri}, we show that \eqref{eq-g*banlambda+b} doesn't hold.
	
	\begin{itemize} 
		\item If \underline{$\lambda \mid n$ and $\lambda m\geqslant 2n$},  we have $0=g^*(b^2)=(a^{\frac{n}{\lambda} }+b)^2=a^{\frac{2n}{\lambda}}$, a contradiction. Thus \eqref{eq-g*banlambda+b} can not happen.
		\item In the following, we assume that \underline{$\lambda \mid n$ and $2n>\lambda m\geqslant n$}. 
	\end{itemize}
	
	(1) For the case $\lambda m=n+j$, $j\equiv 0 \pmod{2\lambda} \text{~and~} 0\leqslant j<n$, set $c=a^{\frac{j}{2\lambda}}b\in H^{\frac{\lambda m+n}{2}}_{}(X;\zz_2)$. We have $cg^\ast(c)=a^mb\ne 0$, which contradicts \Cref{Bredonthm74}. Thus \eqref{eq-g*banlambda+b} can not happen.
	
	(2) Now, let's consider the case $\lambda m=n+j$, $j\equiv \lambda \pmod{2\lambda} \text{~and~} 0\leqslant j<n$. 
	
	When $l\ne n, n+\lambda ,\dots ,n+j$, the coefficient sheaf $\mathcal{H}^l(X;\zz_2)$ is constant with stalk $H^l(X;\zz_2)$ isomorphic to $\zz_2$ or zero. Then $g^*:H^l(X;\zz_2)\to H^l(X;\zz_2)$ is clearly the identity isomorphism, so $\pi_1(B_G)\cong G$ acts trivially on $H^{l}(X;\zz_2)$, and the $E_2$-term of the Leray-Serre spectral sequence associated to the Borel fibration $X\hookrightarrow X_G \to B_G$ is
	\begin{equation}\label{eq-E2klHkHl}
		E_2^{k,l}\cong H^k(B_{G};\zz_2)\otimes H^l(X;\zz_2),  k\geqslant 0, l\ne n, n+\lambda, n+2\lambda,\dots, n+j.
	\end{equation}

To consider the $G$-action on $H^l(X;\zz_2)$ when $l=n, n+\lambda ,\dots ,n+j$, recall that $B_{G}=\mathbb{R}P^\infty $ is a connected CW-complex with one cell in each dimension, 
	\[
	\mathbb{R}P^\infty =e^0\cup e^1 \cup e^2 \cup \cdots.
	\] 
	$E_{G}=S^\infty$ is the universal covering space of $\mathbb{R}P^\infty$, and the corresponding cell decomposition is 
	\[
	S^\infty =e_+^0\cup e_-^0\cup e_+^1\cup e_-^1\cup e_+^2\cup e_-^2\cup \cdots,
	\] 
	with $e_\pm^i$ being the upper and lower hemispheres of the $i$-sphere.  According to \cite[\S5.2.1]{DavisKirk}, the action of $\pi_1(B_{G})\cong\mathbb{Z}_2$ on $S^{\infty}$ gives $C_*(S^{\infty})$ the structure of a $\zz[\zz_2]$-chain complex, where 
	\[
	\mathbb{Z}[\zz_2]=\zz[g]/\langle g^2-1\rangle=\{a_0+a_1g\mid a_0, a_1\in\zz\}
	\]
	denotes the group ring. A basis for the free (rank 1) $\zz[\zz_2]$-module $C_i(S^{\infty})$ is $e_+^i$. With the choice of the basis, the $\zz[\zz_2]$-chain complex $C_*(S^{\infty})$ is isomorphic to 
	\[
	\cdots \to\zz[\zz_2] \to \cdots \xrightarrow[]{\,1-g\,} \zz[\zz_2] \xrightarrow[]{\,1+g\,} \zz[\zz_2] \xrightarrow[]{\,1-g\,} \zz[\zz_2] \to 0.
	\] 
	Let 
	\[
	\tau =1-g^*,~~\sigma =1+g^*. 
	\]   
	The cochain complex ${\rm Hom}_{\zz[\zz_2]}(C_*(S^{\infty}),H^l(X;\zz_2))$ is isomorphic to 
	\[
	\cdots \leftarrow H^l(X;\zz_2) \leftarrow \cdots \xleftarrow[]{~\tau~} H^l(X;\zz_2) \xleftarrow[]{~\sigma~} H^l(X;\zz_2) \xleftarrow[]{~\tau~} H^l(X;\zz_2) \leftarrow 0.
	\]  
	So the $E_2$-term of the Leray-Serre spectral sequence associated to the fibration $X\hookrightarrow X_{G}\to B_{G} $ is given by
	\begin{align*}
		E_2^{k,l} & =H^k(B_{G};\mathcal{H}^l(X;\zz_2))\cong H^k\hskip -1mm\left({\rm Hom}_{\zz[\zz_2]}(C_*(S^{\infty}),H^l(X;\zz_2))\right)\\
		& \cong
		\begin{cases}
			\ker\tau , & k=0,\\
			\ker\tau /\text{im }\sigma , & k>0\text{~even},\\
			\ker\sigma /\text{im }\tau , & k>0\text{~odd}.
		\end{cases}
	\end{align*}
	For $l=n, n+\lambda ,\dots ,n+j$, $H^l(X;\zz_2)\cong\zz_2\oplus\zz_2$ is generated by a basis $a^{\frac{l}{\lambda}}$, $a^{\frac{l-n}{\lambda}}b$. Note that $\tau =\sigma$ and the matrix representation of $\tau$ with the natural basis is 
	\(
	\begin{bmatrix}
		0 & 1\\ 
		0 & 0
	\end{bmatrix}
	\).
	It is easy to see that
	\begin{align}\label{eq-E2{k,l}}
		E_2^{k,l} & \cong
		\begin{cases}
			0 , & k>0\text{~and~}l=n, n+\lambda ,\dots ,n+j,\\
			\mathbb{Z}_2 , & k=0\text{~and~}l=n, n+\lambda ,\dots ,n+j.
		\end{cases}
	\end{align} 
	
	If $X\sim _2\cp^m\times S^n$, $\deg a=\lambda =2$, $\deg b=n$,  $n$ being even implies that $E_2^{k,l}=0$ for $l$ odd. This gives $d_2=0:E_2^{k,l}\to E_{2}^{k+2,l-1}$ and hence $E_2^{*,*}=E_3^{*,*}$. If  $X\sim _2\hp^m\times S^n$, $\lambda=4$, $\lambda$ dividing $n$ implies that $E_r^{k,l}=0$ for $4 \nmid l$. This gives $d_r=0:E_r^{k,l}\to E_{r}^{k+r,l-r+1}$ for $2\leqslant r\leqslant 4$ and hence $E_2^{*,*}=E_5^{*,*}$. That is to say, for $X\sim _2\fp^m\times S^n$, where $\mathbb{F}=\mathbb{R}$, $\mathbb{C}$ or $\mathbb{H}$, we have 
	\begin{equation}\label{eq-E2=Elambda+1}
		E_2^{*,*}=E_{\lambda +1} ^{*,*}. 
	\end{equation}
	
	If $\frac{n}{\lambda}\equiv 1 \pmod{2}$, by \eqref{eq-E2klHkHl}, \eqref{eq-E2=Elambda+1} and the derivation property of the differential,  
\[
d_{\lambda +1}(1\otimes a^{\frac{n}{\lambda}-1})=(\frac{n}{\lambda}-1)(1\otimes a^{\frac{n}{\lambda}-2})d_{\lambda+1}(1\otimes a)=0.
\] 
Note that, $d_{\lambda +1}:E_{\lambda +1}^{k,n+j+\lambda}\to E_{\lambda +1}^{k+\lambda+1,n+j}$ is trivial as $E_{\lambda +1}^{k+\lambda +1,n+j}=0$ (by \eqref{eq-E2{k,l}}) for all $k$, particularly, $d_{\lambda +1}(t^k\otimes a^{\frac{j}{\lambda}+1}b)=0$. By \eqref{eq-E2=Elambda+1}  and \eqref{eq-E2klHkHl},  
\[
E_{\lambda+1}^{k,2n+j}=E_2^{k,2n+j}\cong H^{k}(B_{G};\zz_2)\otimes H^{2n+j}(X;\zz_2) 
\]
is generated by the unique element $t^k\otimes a^{\frac{n+j}{\lambda}}b$. Furthermore, by the multiplicative structure of the spectral sequence, we have $d_{\lambda +1}(t^k\otimes a^{\frac{n+j}{\lambda}}b)=d_{\lambda +1}((t^k\otimes a^{\frac{j}{\lambda}+1}b)(1\otimes a^{\frac{n}{\lambda}-1}))=0$. Consequently, 
	\[
	d_{\lambda +1}:E_{\lambda +1}^{k,2n+j}\to E_{\lambda +1}^{k+\lambda +1,2n+j-\lambda } \text{~is trivial for all ~} k.
	\]
	
%

Set $c=a^{\frac{j-\lambda}{2\lambda}}b$. Then, by \Cref{Bredoncor72}, 
\[
1\otimes cg^\ast (c)=1\otimes a^{\frac{n+j-\lambda}{\lambda}}b\in E_2^{0,2n+j-\lambda} 
\]
is a permanent cocycle. By degree reasons, $t\otimes 1\in E_2^{1,0}$ is a permanent cocycle, therefore $t^k\otimes a^{\frac{n+j-\lambda}{\lambda}}b\in E_2^{k,2n+j-\lambda}$ is also a permanent cocycle for all $k$. 
	
	By \eqref{eq-E2=Elambda+1}, when $\lambda=2$ or $4$, $d_r=0:E_r^{k,l}\to E_{r}^{k+r,l-r+1}$ for $2\leqslant r< \lambda +1$. Moreover, 
\[
d_{\lambda +1}:E_{\lambda +1}^{k,2n+j}\to E_{\lambda +1}^{k+\lambda +1,2n+j-\lambda }
\]
 is trivial for all $k$ and $\lambda=1, 2$ or $4$, hence $t^k\otimes a^{\frac{n+j-\lambda }{\lambda }}b\in E_r^{k,2n+j-\lambda }$ is not hit by any $d_{r}$-coboundaries,  $2\leqslant r\leqslant \lambda +1$. Since $X$ has the mod $2$ cohomology of $\mathbb{F}P^m \times S^n$, for $\lambda m=n+j$, we have $H^l(X;\zz_2)=0$ for $l>2n+j$. As a result, $d_r:E_r^{k-r,2n+j+r-\lambda -1}\to E_{r}^{k,2n+j-\lambda }$ is trivial for $\lambda +1<r$ as $E_2^{k-r,2n+j+r-\lambda -1}=0$. So $t^k\otimes a^{\frac{n+j-\lambda }{\lambda }}b\in E_r^{k,2n+j-\lambda }$ is not hit by any $d_r$-coboundaries, $r\geqslant 2$. Then, $t^k\otimes a^{\frac{n+j-\lambda}{\lambda}}b$ survives to a nontrivial element in $E_{\infty}$. However, this contradicts \Cref{Bredonthm15}. 
Thus \eqref{eq-g*banlambda+b}  doesn't happen. Therefore, the action of $G$ on $H^*(X;\zz_2)$ is trivial.
\end{proof}

As stated in \Cref{lem-acttri}, there are two possible exceptional cases in which we can not prove that the action of $G$ on $H^*(X;\zz_2)$ is trivial. Alternatively, we prove this when $X$ is additionally assumed to have the integral cohomology of $\mathbb{F}P^m \times S^n$ for $\ff=\cc$ or $\hh$.
The following proof is inspired by the discussions in \cite[Theorem 4.5 and Lemma 5.1]{PDey2018}.

\begin{lemma}\label{lem-acttri-Z}
Let $G=\mathbb{Z}_2$ act freely on a finitistic space $X\sim _{\zz} \mathbb{F}P^m \times S^n$, where $\mathbb{F}=\mathbb{C}$ or $\mathbb{H}$. Then the action of $G$ on $H^*(X;\zz_2)$ is trivial.
\end{lemma}
\begin{proof}
The integral cohomology generators of $H^*(X;\zz)\cong H^*(\fp^m\times S^n;\zz)$ are also denoted as $a$  and $b$. Let $g_{\zz}^*$ be the induced integral cohomology homomorphism $H^*(X;\zz)\to H^*(X;\zz)$. Then $g_{\zz}^*$ is an automorphism and preserves degrees as well as cup-length, where for a cohomology class $x$, the cup-length of $x$ is the greatest integer $k$ such that $x^k\ne 0$. Note that, the cup-length of a sum of the integral generators is the sum of the cup-lengths of the individual generators.  
For $m=1$, $n=\lambda$, the mod $2$ cohomology of $X$ is the same as of
\[
S^2\times S^2 ~\text{or~} S^4\times S^4.
\]
If $G$ acts nontrivially on $H^*(X;\zz_2)$, by \Cref{Bredoncor75}, we have $X^{G}$ is non-empty, which contradicts the action being free. Therefore, except for the case when $m=1$, $n=\lambda$, we clearly have 
\[
g_{\zz}^*(a)=\pm a, ~g_{\zz}^*(b) =\pm b.
\]

The integral cohomology group of $X\sim_{\zz}\fp^m\times S^n$ is torsion free for any dimension $l$, so 
\[
H^l(X;\zz_2)\cong H^l(X;\zz)\otimes \zz_2.
\]
Considering the mod $2$ reduction $\phi:H^l(X;\zz)\to H^l(X;\zz_2)$, we have the following commutative diagram
\[
\xymatrix{
H^l(X;\zz) \ar[d]_-{\phi}\ar[r]^-{g_{\zz}^*} & H^l(X;\zz)\ar[d]^-{\phi}\\
H^l(X;\zz_2) \ar[r]_-{g^*} & H^l(X;\zz_2)\\
}
\] 
Thus the mod $2$ cohomology homomorphism $g^*$ is the identity homomorphism.
\end{proof}
\begin{remark}\label{rem-m57RPmS4}
For the exceptional cases $m=5, 7$ in \Cref{thm-Z2faRPmS4},  if $X\sim_{\zz}\rp^m\times S^4$, the trivial action of $G$ on $H^*(X;\zz_2)$  is easily seen as follows.

It is known that 
\begin{align*}
H^*(\rp^{2k+1};\zz) & \cong \zz[a^{}_1, a^{}_2]/\langle 2a^{}_1, a_1^{k+1}, a_2^2, a^{}_1a^{}_2\rangle, \deg a^{}_1= 2, \deg a^{}_2 = 2k + 1.
\end{align*}
Let $g_{\zz}^*: H^*(X;\zz)\to H^*(X;\zz)$ be the induced automorphism. The action of the involution on the generator $b\in H^4(X;\zz)$ coming from $H^4(S^4;\zz)$ must reduce mod 2 to the identity action. 
Otherwise we would have that $g_{\zz}^*(b)= \pm b+a^2_1$ and this can not happen as the class $b\pm a^2_1$ does not square to zero in $H^*(X;\zz)$.  \qed
\end{remark}

Lemmas \ref{lem-acttri}, \ref{lem-acttri-Z} and \Cref{rem-m57RPmS4}  are sufficient for the proof of main Theorems \ref{thm-Z2faRPmS4}, \ref{thm-Z2faCPmS4} and \ref{thm-Z2faHPmS4}. The parts of the main
theorems that are affected by the exceptions of \Cref{lem-acttri} occur only in cases of $\rp^5\times S^4$, $\rp^7\times S^4$, $\cp^3\times S^4$ and $\hp^m\times S^4$, $m\equiv 3\pmod 4$. But with an additional assumption about the trivial action of $G$ on $H^*(X;\zz_2)$ or the integral cohomology of $\fp^m\times S^4$ mentioned in \Cref{lem-acttri-Z} and \Cref{rem-m57RPmS4}, it is possible to show the above cases of the main theorems.

\section{Proofs of the main theorems}\label{sec-proofs}

Let $G=\zz_2$ act freely on a finitistic space $X\sim _2 \fp^m \times S^4$, where $\ff=\rr$, $\cc$ or $\hh$. By \Cref{lem-acttri}, \ref{lem-acttri-Z} and \Cref{rem-m57RPmS4}, $\pi _1(B_G)\cong\zz_2$ acts trivially on $H^*(X)$,  hence, the $E_2$-term of the Leray-Serre spectral sequence associated to the fibration $X\hookrightarrow X_G \to B_G$ has the form
\[
E_2^{k,l}=H^k(B_G)\otimes H^l(X).
\] 
Recall that, \[ H^*(B_G)=\mathbb{Z}_2[t], \text{~where~}\deg t=1. \]

\subsection{Proof of \Cref{thm-Z2faRPmS4}}

Let $G=\zz_2$ act freely on $X\sim _2 \rp^m \times S^4$. Using the K$\ddot{\text{u}}$nneth formula, we observe that, 
\begin{align*}
	H^l(X) & =
	\begin{cases}
		\zz_2, & 0\leqslant l\leqslant \min\{3,m\} \text{~or~} \max\{4,m+1\}\leqslant l\leqslant m+4,\\
		(\zz_2)^2, & 4\leqslant l\leqslant m,\\
		0, & \text{otherwise}.
	\end{cases}
\end{align*}
Let $a\in H^1(X)$ and $b\in H^4(X)$ be the generators of the cohomology algebra of $H^*(X)$, satisfying $a^{m+1}=0$ and $b^2=0$. By degree reasons, $t\otimes 1\in E_2^{1,0}$ is a permanent cocycle and survives to a nontrivial element $x\in E_{\infty}^{1,0}$, i.e., by the edge homomorphism,  
\begin{equation}\label{eq-x=pi*t}
x=\pi ^*(t)\in E_\infty ^{1,0}\subset H^1(X_G). 
\end{equation}

Since $\zz_2$ acts freely on $X$, by \Cref{Bredonthm15}, the spectral sequence does not collapse. Otherwise, we get $H^i(X/G)\ne 0$ for infinitely many values of $i>m+4$. It implies that some differential $d_r:E_r^{k,l}\to E_r^{k+r,l-r+1}$ must be nontrivial. Note that $E_2^{*,*}$ is generated by $t\otimes 1\in E_2^{1,0}$, $1\otimes a\in E_2^{0,1}$ and $1\otimes b\in E_2^{0,4}$. There can only be nontrivial differentials $d_r$ on these generators when $2\leqslant r\leqslant 5$. 
It follows immediately that there are five possibilities for nontrivial differentials on generators:

\begin{enumerate}[(i)]\itemindent=2em
\item $d_2(1\otimes a)\ne 0$.
\item $d_2(1\otimes a)=0$, $d_r(1\otimes b)=0$, $r=2,3,4$ and $d_5(1\otimes b)\ne 0$.
\item $d_2(1\otimes a)=0$, $d_r(1\otimes b)=0$, $r=2,3$ and $d_4(1\otimes b)\ne 0$.
\item $d_2(1\otimes a)=0$, $d_2(1\otimes b)=0$ and $d_3(1\otimes b)\ne 0$.
\item $d_2(1\otimes a)=0$ and $d_2(1\otimes b)\ne 0$.
\end{enumerate}
In the following, we discuss each case separately. 

\textbf{Case(\romannumeral 1)} 
$d_2(1\otimes a)=t^2\otimes 1\ne 0$. 

If $m$ is even, then $a^{m+1}=0$ gives $0=d_2((1\otimes a^m)(1\otimes a))=t^2 \otimes a^m$, a contradiction. Hence $m$ must be odd. There are two possible subcases: either $d_2(1\otimes b)=t^2\otimes a^3\ne 0$ or $d_2(1\otimes b)=0$.

If $d_2(1\otimes b)=t^2\otimes a^3\ne 0$ (in this subcase, $m\geqslant 3$),  by the derivation property of the differential, we have
\begin{align*}
	\begin{cases}
		d_2(1\otimes a^j)=j(t^2\otimes a^{j-1}), & 1\leqslant j\leqslant m, \\
		d_2(1\otimes a^jb)=t^2\otimes a^{j+3}+j(t^2\otimes a^{j-1}b), & 0\leqslant j\leqslant m-3, \\
		d_2(1\otimes a^jb)=j(t^2\otimes a^{j-1}b), & m-2\leqslant j\leqslant m.
	\end{cases}
\end{align*}
Note that 
\begin{align*}
	d_2(1\otimes ab) & =
	\begin{cases}
		t^2\otimes b+t^2\otimes a^4, & m\geqslant 5,\\
		t^2\otimes b, & m=3.
	\end{cases}\\
	d_2d_2(1\otimes ab) & =
	\begin{cases}
		d_2(t^2\otimes b+t^2\otimes a^4), & m\geqslant 5,\\
		d_2(t^2\otimes b), & m=3.
	\end{cases}\\
	& =t^4\otimes a^3\ne 0
\end{align*} 
This contradicts $d_2d_2=0$, thus $d_2(1\otimes b)=0$. By the derivation property of the differential, we have
\begin{align*}
\begin{cases}
d_2(1\otimes a^j)=j(t^2\otimes a^{j-1}), & 1\leqslant j\leqslant m, \\
d_2(1\otimes a^jb)=j(t^2\otimes a^{j-1}b), & 0\leqslant j\leqslant m.
\end{cases}
\end{align*}

The $E_2$-term and $d_2$-differentials look like \Cref{case1-d2}. In all Figures of this paper, we write $t^ka^l, t^ka^{l-4}b$ for $t^k\otimes a^l, t^k\otimes a^{l-4}b\in E_2^{k,l}$ respectively. Each black dot represents a $\zz_2$ summand and the two types of lines (colored by {\red red}, {\color{cyan} cyan}) represent multiplication by $a$ and $b$, and the arrowed line (colored by {\blue blue}) represents a nontrivial differential. In columns $k-2$ and $k$, if there is no arrowed line starting from a black dot, then $d_2$ vanishes on this class.

\begin{figure}[h]
\centering
\DeclareSseqGroup\tower{}{\class(0,0)\foreach \y in {1,...,6}{\class(0,\y) \structline }}
\DeclareSseqGroup\towerfive{}{\class(0,0)\foreach \y in {1,...,5}{\class(0,\y) \structline }}
\DeclareSseqGroup\towerfour{}{\class(0,0)\foreach \y in {1,...,4}{\class(0,\y) \structline }}
\DeclareSseqGroup\towerthree{}{\class(0,0)\foreach \y in {1,...,3}{\class(0,\y) \structline }}
\begin{sseqpage}[classes = fill, class labels = { left = 0.2em }, x tick step = 2, xscale = 1.3, yscale = 0.8, no ticks,class placement transform = {scale = 2}]
	\class[white](0,0)\class[white](7,0)
	\towerfive[struct lines = magenta, classes = black](2,0) \class[black](2,6) \towerthree[struct lines = white, classes = white](2,7)
	\draw[dashed, magenta](1.74,5)--(1.74,6);
	\draw[magenta](2.26,4)--(2.26,5);\draw[dashed, magenta](2.26,5)--(2.26,6);
	\draw[cyan](2,0)--(2.26,4);\draw[cyan](2,1)--(2.26,5);\draw[cyan](2,6)--(2.26,10);
	\foreach \foreach \y in {2,3,4,5}\draw[cyan](2,\y)--(1.87,1+\y);
	\towerthree[struct lines = white, classes = white](2,0) \class[black](2,4) \class[black](2,5)\towerfour[struct lines = magenta, classes = black](2,6)
	
	\towerfive[struct lines = magenta, classes = black](4,0) \class[black](4,6) \towerthree[struct lines = white, classes = white](4,7)
	\draw[dashed, magenta](3.74,5)--(3.74,6);
	\draw[magenta](4.26,4)--(4.26,5);\draw[dashed, magenta](4.26,5)--(4.26,6);
	\draw[cyan](4,0)--(4.26,4);\draw[cyan](4,1)--(4.26,5);\draw[cyan](4,6)--(4.26,10);
	\foreach \foreach \y in {2,3,4,5}\draw[cyan](4,\y)--(3.87,1+\y);
	\towerthree[struct lines = white, classes = white](4,0) \class[black](4,4) \class[black](4,5)\towerfour[struct lines = magenta, classes = black](4,6)
	
	\towerfive[struct lines = magenta, classes = black](6,0) \class[black](6,6) \towerthree[struct lines = white, classes = white](6,7)
	\draw[dashed, magenta](5.74,5)--(5.74,6);
	\draw[magenta](6.26,4)--(6.26,5);\draw[dashed, magenta](6.26,5)--(6.26,6);
	\draw[cyan](6,0)--(6.26,4);\draw[cyan](6,1)--(6.26,5);\draw[cyan](6,6)--(6.26,10);
	\foreach \foreach \y in {2,3,4,5}\draw[cyan](6,\y)--(5.87,1+\y);
	\towerthree[struct lines = white, classes = white](6,0) \class[black](6,4) \class[black](6,5)\towerfour[struct lines = magenta, classes = black](6,6)
	
	\node[background] at (0,-1) {\cdots};
	\node[background] at (2,-1) {k-2\geqslant 0}; \node[background] at (4,-1) {k}; \node[background] at (6,-1) {k+2};
	\foreach \y in {0,1,2,3,4,5}\node[background] at (-1,\y) {\y};
	\node[background] at (-1,5.6) {\vdots};\node[background] at (-1,6) {m};
	\foreach \y in {1,2,3,4}\node[background] at (-1.2,6+\y) {m+\y};
	
	\draw[blue,->](2,1)--(4,0);\draw[blue,->](2,3)--(4,2);\draw[blue,->](2,5)--(4,4);\draw[blue,->](2,6)--(3.0,5.37);
	\draw[blue,->](2,5,2)--(4,4,2);\draw[blue,->](2,6,2)--(3.52,5.37);\draw[blue,->](2,8,2)--(4,7,2);\draw[blue,->](2,10,2)--(4,9,2);
	
	\draw[blue,->](4,1)--(6,0);\draw[blue,->](4,3)--(6,2);\draw[blue,->](4,5)--(6,4);\draw[blue,->](4,6)--(5.0,5.37);
	\draw[blue,->](4,5,2)--(6,4,2);\draw[blue,->](4,6,2)--(5.52,5.37);\draw[blue,->](4,8,2)--(6,7,2);\draw[blue,->](4,10,2)--(6,9,2);
	\node[blue] at (3.2,5.1) {\vdots};\node[blue] at (5.2,5.1) {\vdots};
	
    \foreach \foreach \y in {0,1,2,3,4,5,6,7,8,9,10}
	\draw[lightgray!50, thin](0,\y)--(8,\y);
	
	\node[background] at (1.28,0.05) {t^{k-2}};\node[background] at (1.20,1.05) {t^{k-2}a};\node[background] at (1.10,6.05) {t^{k-2}a^m};
	\node[background] at (3.43,0.05) {t^{k}};\node[background] at (3.35,1.05) {t^{k}a};\node[background] at (3.25,6.05) {t^{k}a^m};
	\node[background] at (5.28,0.05) {t^{k+2}};\node[background] at (5.20,1.05) {t^{k+2}a};\node[background] at (5.10,6.05) {t^{k+2}a^m};

	\node[background] at (2.8,4.05) {t^{k-2}b};\node[background] at (3.0,10.05) {t^{k-2}a^mb};
	\node[background] at (4.65,4.05) {t^{k}b};\node[background] at (4.7,5.05) {t^{k}ab};\node[background] at (4.85,10.05) {t^{k}a^mb};
	\node[background] at (6.8,4.05) {t^{k+2}b};\node[background] at (7.0,10.05) {t^{k+2}a^mb};
\end{sseqpage}
\caption{$E_2$-term and $d_2$-differentials in \textbf{Case(\romannumeral 1)}}\label{case1-d2}
\end{figure}

Since 
\[
E_3^{k,l}=\dfrac{\ker\{d_2:E_2^{k,l}\to E_2^{k+2,l-1}\}}{\text{im}\{d_2:E_2^{k-2,l+1}\to E_2^{k,l}\}},
\]
it is clear from the \Cref{case1-d2} above that 
\begin{align*}
	E_3^{k,l} & =
	\begin{cases}
		\zz_2 , & k=0,1; l=0,2,m+1,m+3,\\
		(\zz_2)^2, & k=0,1; 4\leqslant l\leqslant m \text{~and~} l \text{~even},\\
		0 , & \text{otherwise}.
	\end{cases}
\end{align*}  
Note that $d_r:E_r^{k,l}\to E_r^{k+r,l-r+1}$ is zero for all $r\geqslant 3$ as $E_r^{k+r,l-r+1}=0$, so 
\[
E_3^{*,*}=E_\infty ^{*,*}. 
\]

Since $H^*(X_G)\cong \text{Tot}E_\infty ^{*,*}$, the additive structure of $H^*(X_G)$ is given by
\begin{align*}
	H^j(X_G) & =
	\begin{cases}
		\zz_2 , & 0\leqslant j\leqslant 3\text{~or~}m+1\leqslant j\leqslant m+4,\\
		(\zz_2)^2, & 4\leqslant j\leqslant m,\\
		0 , & j>m+4.
	\end{cases}
\end{align*} 

As $E_\infty ^{2,0}=0$, by \eqref{eq-x=pi*t}, we have $x^2=0$. Notice that, the elements $1\otimes a^2\in E_2^{0,2}$ and $1\otimes b\in E_2^{0,4}$ are permanent cocycles and are not hit by any $d_r$-coboundaries. Hence, they determine nontrivial elements $u\in E_\infty ^{0,2}$ and $v\in E_\infty ^{0,4}$, respectively. We have $u^{\frac{m+1}{2} }=0$ as $a^{m+1}=0$, and $v^2=0$ as $b^2=0$. Thus
\[
\text{Tot}E_\infty ^{*,*}\cong \zz_2[x,u,v]/\langle x^2,u^{\frac{m+1}{2}},v^2 \rangle,
\] 
where $\deg x=1$, $\deg u=2$, $\deg v=4$.

By the identification of the edge homomorphism there exist $y\in H^2(X_G)$ and $z\in H^4(X_G)$ such that $i^*(y)=a^2$ and $i^*(z)=b$, respectively. Notice that $H^j(X_G)=E_{\infty}^{k,j-k}$, where $k=0,1$ and $j-k$ even. Consequently, $y^{\frac{m+1}{2}}\in H^{m+1}(X_G)=E_{\infty}^{0,m+1}$ is represented by $a^{m+1}\in E_{2}^{0,m+1}$ and $z^2\in H^8(X_G)=E_{\infty}^{0,8}$ is represented by $b^2\in E_{2}^{0,8}$. So we have the following relations:
\[
y^{\frac{m+1}{2}}=0,~z^2=0.
\]
Therefore, $H^*(X_G)$  is the graded commutative algebra
\[
\zz_2[x,y,z]/\langle x^2,y^{\frac{m+1}{2}},z^2\rangle,
\] 
where $\deg y=2$, $\deg z=4$ and $m$ is odd. This gives possibility (1) of \Cref{thm-Z2faRPmS4}.

\textbf{Case(\romannumeral 2)}
$d_2(1\otimes a)=0$, $d_r(1\otimes b)=0$, $2\leqslant r\leqslant 4$ and $d_5(1\otimes b)=t^5\otimes 1\ne 0$.

In this case we have $d_r=0$, $2\leqslant r\leqslant 4$, $E_2^{*,*}=E_5^{*,*}$, and
\begin{align*}
	\begin{cases}
		d_5(1\otimes a^j)=0, & 1\leqslant j\leqslant m, \\
		d_5(1\otimes a^jb)=t^5\otimes a^j, & 0\leqslant j\leqslant m.
	\end{cases}
\end{align*}
Furthermore, we have
\[
\xymatrix@R=1mm{
 E_5^{k-5,l+4}\ar[r]^-{d_5} & E_5^{k,l}\ar[r]^-{d_5} & E_5^{k+5,l-4},\\
t^{k-5}\otimes a^lb\ar@{|->}[r]^-{d_5} & t^k\otimes a^l\ar@{|->}[r]^-{d_5} & 0,  \\
t^{k-5}\otimes a^{l+4}\ar@{|->}[r]^-{d_5} & 0, ~t^k\otimes a^{l-4}b\ar@{|->}[r]^-{d_5} & t^{k+5}\otimes a^{l-4}.}
\]
So
\begin{align*}
	E_6^{k,l} & =
	\begin{cases}
		\zz_2 , & 0\leqslant k\leqslant 4;0\leqslant l\leqslant m,\\
		0 , & \text{otherwise}.
	\end{cases}
\end{align*} 
Note that
\(
d_r:E_r^{k,l}\to E_r^{k+r,l-r+1}
\)
is zero for all $r\geqslant 6$ as $E_r^{k+r,l-r+1}=0$, so 
\[
E_6^{*,*}=E_\infty ^{*,*}. 
\]
The additive structure of $H^*(X_G)$ is given by 
\begin{align}
	H^j(X_G) & =
	\begin{cases}
		\zz_2 , & j=0,m+4,\\
		(\zz_2)^2, & j=1,m+3,\\
		(\zz_2)^3, & j=2,m+2,\\
		(\zz_2)^4, & j=3,m+1,\\
		(\zz_2)^5, & 4\leqslant j\leqslant m ~(\text{for~}m\geqslant 4),\\
		0 , & \text{otherwise}.\label{eq-HjXG}
	\end{cases}
\end{align} 

Notice that, the element $1\otimes a\in E_2^{0,1}$ is a permanent cocycle and is not a $d_r$-coboundary. Hence, it determines a nontrivial element $u\in E_\infty ^{0,1}$. As we have remarked, $a^{m+1}=0$, so 
\begin{equation}\label{eq-um+1=0}
	 u^{m+1}=0.
\end{equation}
As $E_\infty ^{5,0}=0$,  by \eqref{eq-x=pi*t}, we have $x^5=0$. Thus
\[
\text{Tot}E_\infty ^{*,*}\cong \zz_2[x,u]/\langle x^5,u^{m+1}\rangle,
\] 
where $\deg u=1$. 

Further, choose $y\in H^1(X_G)$ such that $i^*(y)=a$. By considering the filtration on $H^{m+1}(X_G)$, 
\begin{equation}\label{eq-fil2}
0=F_{m+1}^{m+1}=\cdots =F_{5}^{m+1}\subset F_{4}^{m+1}\subset F_{3}^{m+1} \subset F_{2}^{m+1}\subset F_{1}^{m+1}= F_{0}^{m+1}=H^{m+1}(X_G),
\end{equation}
\vskip -8mm

\hskip 4.1cm$\underbrace{\hskip 11mm}_{E_{\infty}^{4,m-3}}$~\hskip 3mm$\underbrace{\hskip 11mm}_{E_{\infty}^{3,m-2}}$~\hskip 3mm$\underbrace{\hskip 11mm}_{E_{\infty}^{2,m-1}}$~\hskip 3mm$\underbrace{\hskip 11mm}_{E_{\infty}^{1,m}}$~\hskip 3mm

\noindent we get the following relation: 
\[
y^{m+1}=\alpha_1xy^m+\alpha_2x^2y^{m-1}+\alpha _3x^3y^{m-2}+\alpha _4x^4y^{m-3},
\] 
where $\alpha _i\in \zz_2$, $ i=1,\dots ,4$. Therefore, 
\[
H^*(X_G)=\zz_2[x,y]/\langle x^5,y^{m+1}+\alpha_1xy^m+\alpha_2x^2y^{m-1}+\alpha _3x^3y^{m-2}+\alpha _4x^4y^{m-3}\rangle,
\] 
where $\deg y=1$. If $m=1$, then $\alpha_3=\alpha_4=0$. If $m=2$, then $\alpha_4=0$. This gives possibility (2) of \Cref{thm-Z2faRPmS4}.

{\bf In all remaining cases \textbf{Cases(\romannumeral 3)} $\sim$ \textbf{(\romannumeral 5)} there will be classes $u\in E_\infty ^{0,1}$, $y\in H^1(X_G)$ defined as above and the relation \eqref{eq-um+1=0} will be satisfied}. 

\textbf{Case(\romannumeral 3)}
$d_2(1\otimes a)=0$, $d_r(1\otimes b)=0$, $r=2,3$ and $d_4(1\otimes b)=t^4\otimes a\ne 0$.

This case implies that $d_r=0$, $r=2,3$, $E_4^{*,*}=E_2^{*,*}$. So we have
\begin{align*}
	\begin{cases}
		d_4(1\otimes a^j)=0, & 1\leqslant j\leqslant m, \\
		d_4(1\otimes a^jb)=t^4\otimes a^{j+1}, & 0\leqslant j\leqslant m-1, \\
		d_4(1\otimes a^mb)=0.
	\end{cases}
\end{align*}

The $E_4$-term and $d_4$-differentials look like \Cref{case3-d4}. 
\begin{figure}[h]
\centering
\DeclareSseqGroup\tower{}{\class(0,0)\foreach \y in {1,...,6}{\class(0,\y) \structline }}
\DeclareSseqGroup\towerfive{}{\class(0,0)\foreach \y in {1,...,5}{\class(0,\y) \structline }}
\DeclareSseqGroup\towerfour{}{\class(0,0)\foreach \y in {1,...,4}{\class(0,\y) \structline }}
\DeclareSseqGroup\towerthree{}{\class(0,0)\foreach \y in {1,...,3}{\class(0,\y) \structline }}
\begin{sseqpage}[classes = fill, class labels = { left = 0.2em }, x tick step = 2, xscale = 1.3, yscale = 0.9, no ticks,class placement transform = {scale = 2}]
		\class[white](0,0)\class[white](7,0)
	
	\tower[struct lines = magenta, classes = black](2,0)\class[black](2,7)\class[black](2,8)
	\towerthree[struct lines = white, classes = white](2,9)
	\draw[magenta](1.74,6)--(1.74,7);\draw[dashed, magenta](1.74,7)--(1.74,8);
	\draw[magenta](2.26,4)--(2.26,5);\draw[magenta](2.26,5)--(2.26,6);\draw[magenta](2.26,6)--(2.26,7);\draw[dashed, magenta](2.26,7)--(2.26,8);
	\draw[cyan](2,0)--(2.26,4);\draw[cyan](2,1)--(2.26,5);\draw[cyan](2,8)--(2.26,12);
	\foreach \foreach \y in {2,3,4,5,6,7}\draw[cyan](2,\y)--(1.87,1+\y);
	\towerthree[struct lines = white, classes = white](2,0) \class[black](2,4) \class[black](2,5) \class[black](2,6) \class[black](2,7)\towerfour[struct lines = magenta, classes = black](2,8)
	
	\tower[struct lines = magenta, classes = black](4,0)\class[black](4,7)\class[black](4,8)
	\towerthree[struct lines = white, classes = white](4,9)
	\draw[magenta](3.74,6)--(3.74,7);\draw[dashed, magenta](3.74,7)--(3.74,8);
	\draw[magenta](4.26,4)--(4.26,5);\draw[magenta](4.26,5)--(4.26,6);\draw[magenta](4.26,6)--(4.26,7);\draw[dashed, magenta](4.26,7)--(4.26,8);
	\draw[cyan](4,0)--(4.26,4);\draw[cyan](4,1)--(4.26,5);\draw[cyan](4,8)--(4.26,12);
	\foreach \foreach \y in {2,3,4,5,6,7}\draw[cyan](4,\y)--(3.87,1+\y);
	\towerthree[struct lines = white, classes = white](4,0) \class[black](4,4) \class[black](4,5) \class[black](4,6) \class[black](4,7)\towerfour[struct lines = magenta, classes = black](4,8)
	
	\tower[struct lines = magenta, classes = black](6,0)\class[black](6,7)\class[black](6,8)
	\towerthree[struct lines = white, classes = white](6,9)
	\draw[magenta](5.74,6)--(5.74,7);\draw[dashed, magenta](5.74,7)--(5.74,8);
	\draw[magenta](6.26,4)--(6.26,5);\draw[magenta](6.26,5)--(6.26,6);\draw[magenta](6.26,6)--(6.26,7);\draw[dashed, magenta](6.26,7)--(6.26,8);
	\draw[cyan](6,0)--(6.26,4);\draw[cyan](6,1)--(6.26,5);\draw[cyan](6,8)--(6.26,12);
	\foreach \foreach \y in {2,3,4,5,6,7}\draw[cyan](6,\y)--(5.87,1+\y);
	\towerthree[struct lines = white, classes = white](6,0) \class[black](6,4) \class[black](6,5) \class[black](6,6) \class[black](6,7)\towerfour[struct lines = magenta, classes = black](6,8)
	
	\node[background] at (0,-1) {\cdots};
	\node[background] at (2,-1) {k-4\geqslant 0}; \node[background] at (4,-1) {k};\node[background] at (6,-1) {k+4};
	\foreach \y in {0,1,2,3,4,5,6,7}\node[background] at (-1,\y) {\y};
	\node[background] at (-1,7.6) {\vdots};\node[background] at (-1,8) {m};
	\foreach \y in {1,2,3,4}\node[background] at (-1.2,8+\y) {m+\y};
	
	\node[background] at (1.28,0.05) {t^{k-4}};\node[background] at (1.20,1.05) {t^{k-4}a};\node[background] at (1.10,8.05) {t^{k-4}a^m};
	\node[background] at (3.43,0.05) {t^{k}};\node[background] at (3.35,1.05) {t^{k}a};\node[background] at (3.30,2.05) {t^{k}a^2};\node[background] at (3.25,8.05) {t^{k}a^m};
	\node[background] at (5.28,0.05) {t^{k+4}};\node[background] at (5.20,1.05) {t^{k+4}a};\node[background] at (5.10,8.05) {t^{k+4}a^m};
	
	\node[background] at (2.8,4.05) {t^{k-4}b};\node[background] at (3.0,12.05) {t^{k-4}a^mb};
	\node[background] at (4.65,4.05) {t^{k}b};\node[background] at (4.75,5.05) {t^{k}ab};\node[background] at (4.85,12.05) {t^{k}a^mb};
	\node[background] at (6.8,4.05) {t^{k+4}b};\node[background] at (7.0,12.05) {t^{k+4}a^mb};
	
	\draw[blue,->](2,4,2)--(4,1);
	\draw[blue,->](2,5,2)--(4,2);\draw[blue,->](2,6,2)--(4,3);\draw[blue,->](2,7,2)--(4,4);
	\draw[blue,->](2,8,2)--(3.26,5.97);\draw[blue,->](2,9,2)--(3.26,6.97);\draw[blue,->](2,10,2)--(3.26,7.97);
	\draw[blue,->](2,11,2)--(4,8,1);
	
	\draw[blue,->](4,4,2)--(6,1);
	\draw[blue,->](4,5,2)--(6,2);\draw[blue,->](4,6,2)--(6,3);\draw[blue,->](4,7,2)--(6,4);
	\draw[blue,->](4,8,2)--(5.26,5.97);\draw[blue,->](4,9,2)--(5.26,6.97);\draw[blue,->](4,10,2)--(5.26,7.97);
	\draw[blue,->](4,11,2)--(6,8,1);
	\node[blue] at (3.0,6.2) {\vdots};\node[blue] at (5.0,6.2) {\vdots};
	
	\foreach \foreach \y in {0,1,2,3,4,5,6,7,8,9,10,11,12}
	\draw[lightgray!50, thin](0,\y)--(8,\y);
\end{sseqpage}
\caption{$E_4$-term and $d_4$-differentials in \textbf{Case(\romannumeral 3)}}\label{case3-d4}
\end{figure}
Then
\begin{align*}
	E_5^{k,l} & =
	\begin{cases}
		\zz_2 , & k\geqslant 4;l=0,m+4,\\
		\zz_2 , & 0\leqslant k\leqslant 3;0\leqslant l\leqslant m, l=m+4,\\
		0 , & \text{otherwise}.
	\end{cases}
\end{align*} 
Since $1\otimes a$ is a permanent cocycle, by the derivation property of the differential, $d_5(1\otimes a^j)=0$, $1\leqslant j\leqslant m$, and all $d_5: E_5^{k,l}\to E_5^{k+5,l-4}$ is zero by degree reasons. Similarly, 
\(
d_r:E_r^{k,l}\to E_r^{k+r,l-r+1}
\)
is zero for all $6\leqslant r\leqslant m+4$. Thus
\[
E_{m+5}^{*,*}=E_{5}^{*,*}.
\]
Now if $d_{m+5}:E_{m+5}^{0,m+4}\to E_{m+5}^{m+5,0}$ is trivial, then by the multiplicative properties of the spectral sequence, we have $E_{m+5}^{*,*}=E_\infty ^{*,*}$. Therefore the bottom line ($l=0$) and the top line ($l=m+4$) of the spectral sequence survive to $E_\infty $, which reduces to $H^i(X/G)\ne 0$ for all {$i>m+4$}. That contradicts \Cref{Bredonthm15}. Thus, $d_{m+5}:E_{m+5}^{0,m+4}\to E_{m+5}^{m+5,0}$ must be nontrivial. 
It follows immediately that $d_{m+5}:E_{m+5}^{k,m+4}\to E_{m+5}^{k+m+5,0}$ is an isomorphism for all $k$. So
\begin{align*}
	E_{m+6}^{k,l} & =
	\begin{cases}
		\zz_2 , & 4\leqslant k\leqslant m+4;l=0,\\
		\zz_2 , & 0\leqslant k\leqslant 3;0\leqslant l\leqslant m,\\
		0 , & \text{otherwise}.
	\end{cases}
\end{align*}
Note that
\(
d_r:E_r^{k,l}\to E_r^{k+r,l-r+1}
\)
is zero for all $r\geqslant m+6$ as $E_r^{k+r,l-r+1}=0$, so 
\[
E_{\infty}^{*,*}=E_{m+6}^{*,*}. 
\]
It follows that the cohomology groups $H^j(X_G)$ are the same as \eqref{eq-HjXG}.

As $E_\infty ^{m+5,0}=0$,  by \eqref{eq-x=pi*t}, we have $x^{m+5}=0$. Clearly, $x^4u=0$. Combining with \eqref{eq-um+1=0}, then
\[
\text{Tot}E_\infty ^{*,*}\cong \zz_2[x,u]/\langle x^{m+5},u^{m+1},x^4u\rangle.
\] 

Now, choose $y'\in H^1(X_G)$ such that $i^*(y')=a$. By considering the filtration on $H^{m+1}(X_G)$, 
\begin{equation}\label{eq-fil3}
0\subset F_{m+1}^{m+1}=\cdots = F_{4}^{m+1}\subset F_{3}^{m+1} \subset F_{2}^{m+1}\subset F_{1}^{m+1}= F_{0}^{m+1}=H^{m+1}(X_G),
\end{equation}
\vskip -5mm

\hskip 5mm$\underbrace{\hskip 11mm}_{E_{\infty}^{m+1,0}}$~\hskip 26mm$\underbrace{\hskip 11mm}_{E_{\infty}^{3,m-2}}$~\hskip 3mm$\underbrace{\hskip 11mm}_{E_{\infty}^{2,m-1}}$~\hskip 3mm$\underbrace{\hskip 11mm}_{E_{\infty}^{1,m}}$~\hskip 3mm

\noindent we get the following relation: 
\[
(y')^{m+1}=\alpha^{\prime}_1x(y')^m+\alpha^{\prime}_2x^2(y')^{m-1}+\alpha^{\prime}_3x^3(y')^{m-2}+\alpha^{\prime}_4x^{m+1},
\] 
where $\alpha^{\prime}_i\in \zz_2$, $ i=1,\dots ,4$. By considering the filtration on $H^{5}(X_G)$, 
\begin{equation*}
	0\subset F_{5}^{5}=F_{4}^{5}\subset  F_{3}^{5} \subset F_{2}^{5}\subset F_{1}^{5}\subset F_{0}^{5}=H^{5}(X_G),
\end{equation*}
\vskip -5mm

\hskip 32mm$\underbrace{\hskip 7mm}_{E_{\infty}^{5,0}}$~
\hskip 10mm$\underbrace{\hskip 9mm}_{E_{\infty}^{3,2}}$~
\hskip 1mm$\underbrace{\hskip 9mm}_{E_{\infty}^{2,3}}$~
\hskip 1mm$\underbrace{\hskip 9mm}_{E_{\infty}^{1,4}}$~
\hskip 1mm $\underbrace{\hskip 9mm}_{E_{\infty}^{0,5}}$

\noindent we can write $x^4y'$ as: 
\[
x^4y'=\beta x^5, \beta \in \zz_2.
\] 
By choosing a particular 
\begin{equation}\label{eq-yy'}
	y=y'+\beta x,
\end{equation}
the above relations can be simplified as
\begin{align*}
	y^{m+1} & =\alpha_1xy^m+\alpha_2x^2y^{m-1}+\alpha _3x^3y^{m-2}+\alpha _4x^{m+1},\\
	x^4y & =0,
\end{align*}
where $\alpha_i~(i=1,\dots ,4)\in \zz_2$. Thus, $H^*(X_G)$ is the graded commutative algebra $\zz_2[x,y]/I$, where $I$ is the ideal given by
\[
I=\langle x^{m+5},y^{m+1}+\alpha_1xy^m+\alpha_2x^2y^{m-1}+\alpha _3x^3y^{m-2}+\alpha _4x^{m+1},x^4y\rangle,
\]
If $m=1$, then $\alpha_3=\alpha_4=0$. If $m=2$, then $\alpha_4=0$. This gives possibility (3) of \Cref{thm-Z2faRPmS4}.

\textbf{Case(\romannumeral 4)}
$d_2(1\otimes a)=0$, $d_2(1\otimes b)=0$ and $d_3(1\otimes b)=t^3\otimes a^2\ne 0$.

Obviously, $m\geqslant 2$, $d_2=0$ and $E_{3}^{*,*}=E_2^{*,*}$. Since
\begin{align*}
	\begin{cases}
		d_3(1\otimes a^j)=0, & 1\leqslant j\leqslant m, \\
		d_3(1\otimes a^jb)=t^3\otimes a^{j+2}, & 0\leqslant j\leqslant m-2, \\
		d_3(1\otimes a^jb)=0, & j=m-1,m,
	\end{cases}
\end{align*}
the $E_3$-term and $d_3$-differentials look like \Cref{case4-d3}. 
\begin{figure}[h]
\centering
\DeclareSseqGroup\tower{}{\class(0,0)\foreach \y in {1,...,6}{\class(0,\y) \structline }}
\DeclareSseqGroup\towerfive{}{\class(0,0)\foreach \y in {1,...,5}{\class(0,\y) \structline }}
\DeclareSseqGroup\towerfour{}{\class(0,0)\foreach \y in {1,...,4}{\class(0,\y) \structline }}
\DeclareSseqGroup\towerthree{}{\class(0,0)\foreach \y in {1,...,3}{\class(0,\y) \structline }}
\begin{sseqpage}[classes = fill, class labels = { left = 0.2em }, x tick step = 2, xscale = 1.3, yscale = 0.8, no ticks,class placement transform = {scale = 2}]
	\class[white](0,0)\class[white](7,0)
	
	\tower[struct lines = magenta, classes = black](2,0)\class[black](2,7)\towerthree[struct lines = white, classes = white](2,8)
	\draw[dashed, magenta](1.74,6)--(1.74,7);
	\draw[magenta](2.26,4)--(2.26,5);\draw[magenta](2.26,5)--(2.26,6);\draw[dashed, magenta](2.26,6)--(2.26,7);
	\draw[cyan](2,0)--(2.26,4);\draw[cyan](2,1)--(2.26,5);\draw[cyan](2,7)--(2.26,11);
	\foreach \foreach \y in {2,3,4,5,6}\draw[cyan](2,\y)--(1.87,1+\y);
	\towerthree[struct lines = white, classes = white](2,0)\class[black](2,4) \class[black](2,5)\class[black](2,6)\towerfour[struct lines = magenta, classes = black](2,7)
	
	\tower[struct lines = magenta, classes = black](4,0)\class[black](4,7)\towerthree[struct lines = white, classes = white](4,8)
	\draw[dashed, magenta](3.74,6)--(3.74,7);
	\draw[magenta](4.26,4)--(4.26,5);\draw[magenta](4.26,5)--(4.26,6);\draw[dashed, magenta](4.26,6)--(4.26,7);
	\draw[cyan](4,0)--(4.26,4);\draw[cyan](4,1)--(4.26,5);\draw[cyan](4,7)--(4.26,11);
	\foreach \foreach \y in {2,3,4,5,6}\draw[cyan](4,\y)--(3.87,1+\y);
	\towerthree[struct lines = white, classes = white](4,0)\class[black](4,4) \class[black](4,5)\class[black](4,6)\towerfour[struct lines = magenta, classes = black](4,7)
	
	\tower[struct lines = magenta, classes = black](6,0)\class[black](6,7)\towerthree[struct lines = white, classes = white](6,8)
	\draw[dashed, magenta](5.74,6)--(5.74,7);
	\draw[magenta](6.26,4)--(6.26,5);\draw[magenta](6.26,5)--(6.26,6);\draw[dashed, magenta](6.26,6)--(6.26,7);
	\draw[cyan](6,0)--(6.26,4);\draw[cyan](6,1)--(6.26,5);\draw[cyan](6,7)--(6.26,11);
	\foreach \foreach \y in {2,3,4,5,6}\draw[cyan](6,\y)--(5.87,1+\y);
	\towerthree[struct lines = white, classes = white](6,0)\class[black](6,4) \class[black](6,5)\class[black](6,6)\towerfour[struct lines = magenta, classes = black](6,7)
	
	\node[background] at (0,-1) {\cdots};
	\node[background] at (2,-1) {k-3\geqslant 0}; \node[background] at (4,-1) {k};\node[background] at (6,-1) {k+3};
	\foreach \y in {0,1,2,3,4,5,6}\node[background] at (-1,\y) {\y};
	\node[background] at (-1,6.6) {\vdots};\node[background] at (-1,7) {m};
	\foreach \y in {1,2,3,4}\node[background] at (-1.2,7+\y) {m+\y};
	
	\node[background] at (1.28,0.05) {t^{k-3}};\node[background] at (1.20,1.05) {t^{k-3}a};\node[background] at (1.10,7.05) {t^{k-3}a^m};
	\node[background] at (3.43,0.05) {t^{k}};\node[background] at (3.35,1.05) {t^{k}a};\node[background] at (3.30,2.05) {t^{k}a^2};\node[background] at (3.25,7.05) {t^{k}a^m};
	\node[background] at (5.28,0.05) {t^{k+3}};\node[background] at (5.20,1.05) {t^{k+3}a};\node[background] at (5.10,7.05) {t^{k+3}a^m};
	
	\node[background] at (2.8,4.05) {t^{k-3}b};\node[background] at (3.0,11.05) {t^{k-3}a^mb};
	\node[background] at (4.65,4.05) {t^{k}b};\node[background] at (4.75,5.05) {t^{k}ab};\node[background] at (4.85,11.05) {t^{k}a^mb};
	\node[background] at (6.8,4.05) {t^{k+3}b};\node[background] at (7.0,11.05) {t^{k+3}a^mb};
	
	\draw[blue,->](2,4,2)--(4,2);\draw[blue,->](2,5,2)--(4,3);\draw[blue,->](2,6,2)--(4,4);
	\draw[blue,->](2,7,2)--(3.26,5.65);\draw[blue,->](2,8,2)--(3.26,6.65);
	\draw[blue,->](2,9,2)--(4,7,1);
	
	\draw[blue,->](4,4,2)--(6,2);\draw[blue,->](4,5,2)--(6,3);\draw[blue,->](4,6,2)--(6,4);
	\draw[blue,->](4,7,2)--(5.26,5.65);\draw[blue,->](4,8,2)--(5.26,6.65);
	\draw[blue,->](4,9,2)--(6,7,1);
	\node[blue] at (3.0,5.7) {\vdots};\node[blue] at (5.0,5.7) {\vdots};
	
	\foreach \foreach \y in {0,1,2,3,4,5,6,7,8,9,10,11}
	\draw[lightgray!50, thin](0,\y)--(8,\y);
\end{sseqpage}
\caption{$E_3$-term and $d_3$-differentials in \textbf{Case(\romannumeral 4)}}\label{case4-d3}
\end{figure}
Then
\begin{equation}\label{eq-case4-E4}
	E_4^{k,l} =
	\begin{cases}
		\zz_2 , & k\geqslant 3;l=0,1,m+3,m+4,\\
		\zz_2 , & 0\leqslant k\leqslant 2;0\leqslant l\leqslant m, l=m+3,m+4,\\
		0 , & \text{otherwise}.
	\end{cases}
\end{equation} 
Consider the bidegrees of $E_4^{*,*}$, \(d_r:E_r^{k,l}\to E_r^{k+r,l-r+1}\)
is zero for all $4\leqslant r\leqslant m+2$. So 
\begin{equation}\label{eq-case4-Em+3}
E_{m+3}^{k,l}=E_{4}^{k,l}, \text{ for all } k, l.
\end{equation}
The differential $d_r:E_r^{0,m+3}\to E_r^{r,m+4-r}$($r\geqslant m+3$) can only be nontrivial when $r=m+3$ or $m+4$. If $d_r:E_r^{0,m+3}\to E_r^{r,m+4-r}$ is trivial for $r=m+3$ and $r=m+4$, then $d_r=0: E_r^{k,m+3}\to E_r^{k+r,m+4-r}$ for any $k$, $r=m+3$ and $r=m+4$. Thus $E_{m+3}^{*,*}=E_\infty ^{*,*}$, at least two lines of the spectral sequence survive to $E_\infty $, which contradicts \Cref{Bredonthm15}. Thus, we get two possibilities:  
\begin{enumerate}[(\romannumeral 4.1)]\itemindent=2em
	\item $d_{m+3}: E_{m+3}^{0,m+3}\to E_{m+3}^{m+3,1}$ is nontrivial.
	\item $d_{m+3}: E_{m+3}^{0,m+3}\to E_{m+3}^{m+3,1}$ is trivial and $d_{m+4}:E_{m+4}^{0,m+3}\to E_{m+4}^{m+4,0}$ is nontrivial.
\end{enumerate}

\textbf{Subcase(\romannumeral 4.1)} 
If $d_{m+3}:E_{m+3}^{0,m+3}\to E_{m+3}^{m+3,1}$ is nontrivial, then $d_{m+3}(1\otimes a^{m-1}b)=t^{m+3}\otimes a$, and 
\[
d_{m+3}:E_{m+3}^{k,l}\to E_{m+3}^{k+m+3,l-m-2}
\] 
is an isomorphism for all $k$ and $l=m+3$ and a trivial homomorphism otherwise. Consequently, 
\begin{align*}
	E_{m+4}^{k,l} & =
	\begin{cases}
		\zz_2 , & k\geqslant m+3;l=0,m+4,\\
		\zz_2 , & 3\leqslant k\leqslant m+2;l=0,1,m+4,\\
		\zz_2 , & 0\leqslant k\leqslant 2;0\leqslant l\leqslant m, l=m+4,\\
		0 , & \text{otherwise}.
	\end{cases}
\end{align*} 
The differential $d_r:E_r^{0,m+4}\to E_r^{r,m+5-r}$($r\geqslant m+4$) can only be nontrivial
when $r=m+5$. If $d_{m+5}:E_{m+5}^{0,m+4}\to E_{m+5}^{m+5,0}$ is trivial, then $d_{m+5}=0:E_{m+5}^{k,m+4}\to E_{m+5}^{k+m+5,0}$ for any $k$. Thus $E_{m+4}^{*,*}=E_{\infty}^{*,*}$. Therefore the bottom line ($l=0$) and the top line ($l=m+4$) of the spectral sequence survive to $E_\infty $, which contradicts \Cref{Bredonthm15}. Therefore, the differential $d_{m+5}:E_{m+5}^{0,m+4}\to E_{m+5}^{m+5,0}$ is nontrivial. Then 
\(
d_{m+5}:E_{m+5}^{k,l}\to E_{m+5}^{k+m+5,l-m-4}
\) 
is an isomorphism for all $k$ and $l=m+4$ and a trivial homomorphism otherwise. Consequently, 
\begin{equation}\label{eq-Einfty4.1}
	E_{m+6}^{k,l} =
	\begin{cases}
		\zz_2 , & k=m+3, m+4;l=0,\\
		\zz_2 , & 3\leqslant k\leqslant m+2;l=0,1,\\
		\zz_2 , & 0\leqslant k\leqslant 2;0\leqslant l\leqslant m,\\
		0 , & \text{otherwise}.
	\end{cases}
\end{equation} 
Note that
\(
d_r:E_r^{k,l}\to E_r^{k+r,l-r+1}
\)
is zero for all $r\geqslant m+6$ as $E_r^{k+r,l-r+1}=0$, so 
\[
E_{m+6}^{*,*}=E_\infty ^{*,*}. 
\]
We observe that the cohomology groups $H^j(X_G)$ are the same as \eqref{eq-HjXG}.

As $E_\infty ^{m+5,0}=0$, by \eqref{eq-x=pi*t}, we have $x^{m+5}=0$. Clearly, $x^3u^2=0, x^{m+3}u=0$. Combining with \eqref{eq-um+1=0}, then  
\[
\text{Tot}E_\infty ^{*,*}\cong \zz_2[x,u]/\langle x^{m+5},u^{m+1},x^3u^2,x^{m+3}u \rangle.
\] 

Similar to the discussion of the filtration \eqref{eq-fil2} or \eqref{eq-fil3} and the particular choice of $y$ in \eqref{eq-yy'}, consider \eqref{eq-Einfty4.1}, we get the following relations: 
\begin{align*}
y^{m+1} & =\alpha_1xy^m+\alpha_2x^2y^{m-1}+\alpha _3x^my+\alpha _4x^{m+1},\\
x^3y^2 & =\beta _1x^4y+\beta_2 x^5,\\
x^{m+3}y & =0,
\end{align*}
where $\alpha_i~(i=1,\dots ,4)$, $\beta_1, \beta_2 \in \zz_2$. So the graded commutative algebra $H^*(X_G)$ is  $\zz_2[x,y]/I$, where $I$ is the ideal given by
\begin{align*}
	I=\langle x^{m+5},y^{m+1}+\alpha_1xy^m+\alpha_2x^2y^{m-1}+\alpha _3x^my+\alpha _4x^{m+1}, x^3y^2+\beta _1x^4y+\beta_2 x^5,x^{m+3}y \rangle,
\end{align*}
where $m\geqslant 2$. If $m=2$, then $\alpha_3=0$. This gives possibility (4) of \Cref{thm-Z2faRPmS4}.

\textbf{Subcase(\romannumeral 4.2)}
If $d_{m+3}: E_{m+3}^{0,m+3}\to E_{m+3}^{m+3,1}$ is trivial and $d_{m+4}:E_{m+4}^{0,m+3}\to E_{m+4}^{m+4,0}$ is nontrivial, then 
\begin{align}\label{eq-case4-dm+3}
& d_{m+3}=0: E_{m+3}^{k,l}\to E_{m+3}^{k+m+3,l-m-2}, \text{for any } k, l, \\
& d_{m+4}(1\otimes a^{m-1}b)=t^{m+4}\otimes 1, \nonumber\\
& d_{m+4}(1\otimes a^mb)=t^{m+4}\otimes a. \nonumber
\end{align}
Furthermore, we obtain that
\begin{equation}\label{eq-case4-dm+4}
d_{m+4}:E_{m+4}^{k,l}\to E_{m+4}^{k+m+4,l-m-3}
\end{equation}
is an isomorphism for all $k$, $l=m+3,m+4$ and a trivial homomorphism otherwise. Consequently, by \eqref{eq-case4-Em+3}, \eqref{eq-case4-dm+3} and \eqref{eq-case4-dm+4}, we have
\begin{align*}
	E_{m+5}^{k,l} & =
	\begin{cases}
		\zz_2 , & 3\leqslant k\leqslant m+3;l=0,1,\\
		\zz_2 , & 0\leqslant k\leqslant 2;0\leqslant l\leqslant m,\\
		0 , & \text{otherwise}.
	\end{cases}
\end{align*} 
Note that
\(
d_r:E_r^{k,l}\to E_r^{k+r,l-r+1}
\)
is zero for all $r\geqslant m+5$ as $E_r^{k+r,l-r+1}=0$, so 
\[
E_{m+5}^{*,*}=E_\infty ^{*,*}. 
\]
We observe that the cohomology groups $H^j(X_G)$ are the same as \eqref{eq-HjXG}.

As $E_\infty ^{m+4,0}=0$, by \eqref{eq-x=pi*t}, we have $x^{m+4}=0$.  Clearly, $x^3u^2=0$. Combining with \eqref{eq-um+1=0}, then
\[
\text{Tot}E_\infty ^{*,*}\cong \zz_2[x,u]/\langle x^{m+4},u^{m+1},x^3u^2\rangle,
\] 
The graded commutative algebra $H^*(X_G)$ is $\zz_2[x,y]/I$, where $I$ is the ideal given by
\[
I=\langle x^{m+4},y^{m+1}+\alpha_1xy^m+\alpha_2x^2y^{m-1}+\alpha _3x^my+\alpha _4x^{m+1},x^3y^2+\beta _1x^4y+\beta_2 x^5\rangle,
\] 
where $m\geqslant 2$ and $\alpha _i~(i=1,\dots ,4)$, $\beta_1, \beta_2\in \zz_2$. If $m=2$, then $\alpha_3=0$. This gives possibility (5) of \Cref{thm-Z2faRPmS4}.

\textbf{Case(\romannumeral 5)}
$d_2(1\otimes a)=0$ and $d_2(1\otimes b)=t^2\otimes a^3\ne 0$.

Obviously, $m\geqslant 3$. We have
\begin{align*}
	\begin{cases}
		d_2(1\otimes a^j)=0, & 1\leqslant j\leqslant m, \\
		d_2(1\otimes a^jb)=t^2\otimes a^{j+3}, & 0\leqslant j\leqslant m-3, \\
		d_2(1\otimes a^jb)=0, & m-2\leqslant j\leqslant m.
	\end{cases}
\end{align*}
The $E_2$-term and $d_2$-differentials look like \Cref{case5-d2}. 
\begin{figure}[h]
\centering
\DeclareSseqGroup\tower{}{\class(0,0)\foreach \y in {1,...,6}{\class(0,\y) \structline }}
\DeclareSseqGroup\towerfive{}{\class(0,0)\foreach \y in {1,...,5}{\class(0,\y) \structline }}
\DeclareSseqGroup\towerfour{}{\class(0,0)\foreach \y in {1,...,4}{\class(0,\y) \structline }}
\DeclareSseqGroup\towerthree{}{\class(0,0)\foreach \y in {1,...,3}{\class(0,\y) \structline }}
\begin{sseqpage}[classes = fill, class labels = { left = 0.2em }, x tick step = 2, xscale = 1.3, yscale = 0.8, no ticks,class placement transform = {scale = 2}]
	\class[white](0,0)\class[white](7,0)
	
	\towerfive[struct lines = magenta, classes = black](2,0)\class[black](2,6)\towerthree[struct lines = white, classes = white](2,7)
	\draw[dashed, magenta](1.74,5)--(1.74,6);
	\draw[magenta](2.26,4)--(2.26,5);\draw[dashed, magenta](2.26,5)--(2.26,6);
	\draw[cyan](2,0)--(2.26,4);\draw[cyan](2,1)--(2.26,5);\draw[cyan](2,6)--(2.26,10);
	\foreach \foreach \y in {2,3,4,5}\draw[cyan](2,\y)--(1.87,1+\y);
	\towerthree[struct lines = white, classes = white](2,0)\class[black](2,4)\class[black](2,5)\towerfour[struct lines = magenta, classes = black](2,6)
	
	\towerfive[struct lines = magenta, classes = black](4,0)\class[black](4,6)\towerthree[struct lines = white, classes = white](4,7)
	\draw[dashed, magenta](3.74,5)--(3.74,6);
	\draw[magenta](4.26,4)--(4.26,5);\draw[dashed, magenta](4.26,5)--(4.26,6);
	\draw[cyan](4,0)--(4.26,4);\draw[cyan](4,1)--(4.26,5);\draw[cyan](4,6)--(4.26,10);
	\foreach \foreach \y in {2,3,4,5}\draw[cyan](4,\y)--(3.87,1+\y);
	\towerthree[struct lines = white, classes = white](4,0)\class[black](4,4)\class[black](4,5)\towerfour[struct lines = magenta, classes = black](4,6)
	
	\towerfive[struct lines = magenta, classes = black](6,0)\class[black](6,6)\towerthree[struct lines = white, classes = white](6,7)
	\draw[dashed, magenta](5.74,5)--(5.74,6);
	\draw[magenta](6.26,4)--(6.26,5);\draw[dashed, magenta](6.26,5)--(6.26,6);
	\draw[cyan](6,0)--(6.26,4);\draw[cyan](6,1)--(6.26,5);\draw[cyan](6,6)--(6.26,10);
	\foreach \foreach \y in {2,3,4,5}\draw[cyan](6,\y)--(5.87,1+\y);
	\towerthree[struct lines = white, classes = white](6,0)\class[black](6,4)\class[black](6,5)\towerfour[struct lines = magenta, classes = black](6,6)
	
	\node[background] at (0,-1) {\cdots};
	\node[background] at (2,-1) {k-2\geqslant 0}; \node[background] at (4,-1) {k};\node[background] at (6,-1) {k+2};
	\foreach \y in {0,1,2,3,4,5}\node[background] at (-1,\y) {\y};
	\node[background] at (-1,5.6) {\vdots};\node[background] at (-1,6) {m};
	\foreach \y in {1,2,3,4}\node[background] at (-1.2,6+\y) {m+\y};
	
	\node[background] at (1.28,0.05) {t^{k-2}};\node[background] at (1.20,1.05) {t^{k-2}a};\node[background] at (1.10,6.05) {t^{k-2}a^m};
	\node[background] at (3.43,0.05) {t^{k}};\node[background] at (3.35,1.05) {t^{k}a};\node[background] at (3.30,2.05) {t^{k}a^2};\node[background] at (3.25,6.05) {t^{k}a^m};
	\node[background] at (5.28,0.05) {t^{k+2}};\node[background] at (5.20,1.05) {t^{k+2}a};\node[background] at (5.10,6.05) {t^{k+2}a^m};
	
	\node[background] at (2.8,4.05) {t^{k-2}b};\node[background] at (3.0,10.05) {t^{k-2}a^mb};
	\node[background] at (4.65,4.05) {t^{k}b};\node[background] at (4.75,5.05) {t^{k}ab};\node[background] at (4.85,10.05) {t^{k}a^mb};
	\node[background] at (6.8,4.05) {t^{k+2}b};\node[background] at (7.0,10.05) {t^{k+2}a^mb};
	
	\draw[blue,->](2,4,2)--(4,3);\draw[blue,->](2,5,2)--(4,4);
	\draw[blue,->](2,6,2)--(3.26,5.32);
	\draw[blue,->](2,7,2)--(4,6);
	
	\draw[blue,->](4,4,2)--(6,3);\draw[blue,->](4,5,2)--(6,4);
	\draw[blue,->](4,6,2)--(5.26,5.32);
	\draw[blue,->](4,7,2)--(6,6);
	\node[blue] at (3.0,5.2) {\vdots};\node[blue] at (5.15,5.2) {\vdots};
	
	\foreach \foreach \y in {0,1,2,3,4,5,6,7,8,9,10}
	\draw[lightgray!50, thin](0,\y)--(8,\y);
\end{sseqpage}
\caption{$E_2$-term and $d_2$-differentials in \textbf{Case(\romannumeral 5)}}\label{case5-d2}
\end{figure}
Then
\begin{equation}\label{eq-case5-E3}
	E_3^{k,l} =
	\begin{cases}
		\zz_2 , & k\geqslant 2;l=0,1,2,m+2,m+3,m+4,\\
		\zz_2 , & k=0,1;0\leqslant l\leqslant m, l=m+2,m+3,m+4,\\
		0 , & \text{otherwise}.
	\end{cases}
\end{equation} 
Clearly,  
\(
d_r:E_r^{k,l}\to E_r^{k+r,l-r+1}
\)
is zero for all $3\leqslant r\leqslant m$. So 
\begin{equation}\label{eq-case5-Em+1}
	E_3^{k,l}=E_{m+1}^{k,l}, \text{ for all } k, l.
\end{equation}
The differential $d_r:E_r^{0,m+2}\to E_r^{r,m+3-r}$($r\geqslant m+1$) can only be nontrivial when
$r=m+1,m+2,m+3$. If $d_r:E_r^{0,m+2}\to E_r^{r,m+3-r}$ is trivial for $r=m+1,m+2,m+3$, then $d_r=0:E_r^{k,m+2}\to E_r^{k+r,m+3-r}$ for any $k$, $r=m+1$, $m+2$ and $m+3$. Thus $E_{m+1}^{*,*}=E_\infty ^{*,*}$, at least two lines of the spectral sequence survive to $E_\infty $, which contradicts \Cref{Bredonthm15}. Therefore, we get the following subcases: 
\begin{enumerate}[(\romannumeral 5.1)]\itemindent=2em
	\item $d_{m+1}:E_{m+1}^{0,m+2}\to E_{m+1}^{m+1,2}$ is nontrivial.
	\item $d_{m+1}=0:E_{m+1}^{0,m+2}\to E_{m+1}^{m+1,2}$ and $d_{m+2}:E_{m+2}^{0,m+2}\to E_{m+2}^{m+2,1}$ is nontrivial.
	\item $d_r=0:E_{r}^{0,m+2}\to E_{r}^{r,m+3-r}$, $r=m+1,m+2$ and $d_{m+3}:E_{m+3}^{0,m+2}\to E_{m+3}^{m+3,0}$ is nontrivial.
\end{enumerate}

\textbf{Subcase(\romannumeral 5.1)}
If $d_{m+1}:E_{m+1}^{0,m+2}\to E_{m+1}^{m+1,2}$ is nontrivial, then $d_{m+1}(1\otimes a^{m-2}b)=t^{m+1}\otimes a^2$, and 
\(
d_{m+1}:E_{m+1}^{k,l}\to E_{m+1}^{k+m+1,l-m}
\)
is an isomorphism for all $k$ and $l=m+2$ and a trivial homomorphism otherwise. Consequently, 
\begin{equation}\label{eq-case5.1-Em+2}
	E_{m+2}^{k,l} =
	\begin{cases}
		\zz_2 , & k\geqslant m+1;l=0,1,m+3,m+4,\\
		\zz_2 , & 2\leqslant k\leqslant m;l=0,1,2,m+3,m+4,\\
		\zz_2 , & k=0,1; 0\leqslant l\leqslant m, l=m+3,m+4,\\
		0 , & \text{otherwise}.
	\end{cases}
\end{equation} 
Clearly,  
\(
d_{m+2}:E_{m+2}^{k,l}\to E_{m+2}^{k+m+2,l-m-1}
\)
is zero by degree reasons. So 
\begin{equation}\label{eq-case5.1-Em+3}
	E_{m+2}^{k,l}=E_{m+3}^{k,l}, \text{ for all } k, l.
\end{equation}
The differential $d_r:E_r^{0,m+3}\to E_r^{r,m+4-r}$($r\geqslant m+3$) can only be nontrivial when
$r=m+3,m+4$. If $d_r:E_r^{0,m+3}\to E_r^{r,m+4-r}$ is trivial for $r=m+3,m+4$, then $d_r=0:E_r^{k,m+3}\to E_r^{k+r,m+4-r}$ for any $k$, $r=m+3$ and $m+4$. Thus $E_{m+3}^{*,*}=E_\infty ^{*,*}$, at least two lines of the spectral sequence survive to infinity, which contradicts \Cref{Bredonthm15}. Thus, we get two possibilities:  
\begin{enumerate}[(\romannumeral 5.1.1)]\itemindent=2em
	\item $d_{m+3}:E_{m+3}^{0,m+3}\to E_{m+3}^{m+3,1}$ is nontrivial.
	\item[(\romannumeral 5.1.2)] $d_{m+3}=0:E_{m+3}^{0,m+3}\to E_{m+3}^{m+3,1}$ and $d_{m+4}:E_{m+4}^{0,m+3}\to E_{m+4}^{m+4,0}$ is nontrivial.
\end{enumerate}

\textbf{Subcase(\romannumeral 5.1.1)}
If $d_{m+3}:E_{m+3}^{0,m+3}\to E_{m+3}^{m+3,1}$ is nontrivial, then 
\(
d_{m+3}:E_{m+3}^{k,l}\to E_{m+3}^{k+m+3,l-m-2}
\) 
is an isomorphism for all $k$ and $l=m+3$ and a trivial homomorphism otherwise. Consequently, 
\begin{align*}
	E_{m+4}^{k,l} & =
	\begin{cases}
		\zz_2 , & k\geqslant m+3;l=0,m+4,\\
		\zz_2 , & k=m+1, m+2;l=0,1,m+4,\\
		\zz_2 , & 2\leqslant k\leqslant m;l=0,1,2,m+4,\\
		\zz_2 , & k=0,1;0\leqslant l\leqslant m,l=m+4,\\
		0 , & \text{otherwise}.
	\end{cases}
\end{align*}
The differential $d_r:E_r^{0,m+4}\to E_r^{r,m+5-r}$($r\geqslant m+4$) can only be nontrivial when $r=m+5$. If $d_{m+5}:E_{m+5}^{0,m+4}\to E_{m+5}^{m+5,0}$ is trivial, then $d_{m+5}=0:E_{m+5}^{k,m+4}\to E_{m+5}^{k+m+5,0}$ for any $k$. Thus $E_{m+4}^{*,*}=E_\infty ^{*,*}$, the bottom line ($l=0$) and the top line ($l=m+4$) of the spectral sequence survive to $E_\infty $, which contradicts \Cref{Bredonthm15}. Therefore, $d_{m+5}:E_{m+5}^{0,m+4}\to E_{m+5}^{m+5,0}$ must be nontrivial. Then 
\(
d_{m+5}:E_{m+5}^{k,l}\to E_{m+5}^{k+m+5,l-m-4}
\) 
is an isomorphism for all $k$ and $l=m+4$ and a trivial homomorphism otherwise. Consequently,  
\begin{equation}\label{eq-Einfty5.1.1}
	E_{m+6}^{k,l} =
	\begin{cases}
		\zz_2 , & k=m+3,m+4;l=0,\\
		\zz_2 , & k=m+1,m+2;l=0,1,\\
		\zz_2 , & 2\leqslant k\leqslant m;l=0,1,2,\\
		\zz_2 , & k=0,1;0\leqslant l\leqslant m,\\
		0 , & \text{otherwise}.
	\end{cases}
\end{equation} 
Note that
\(
d_r:E_r^{k,l}\to E_r^{k+r,l-r+1}
\)
is zero for all $r\geqslant m+6$ as $E_r^{k+r,l-r+1}=0$, so 
\[
E_{m+6}^{*,*}=E_\infty ^{*,*}. 
\]
We observe that the cohomology groups $H^j(X_G)$ are the same as \eqref{eq-HjXG}.

As $E_\infty ^{m+5,0}=0$, by \eqref{eq-x=pi*t}, we have $x^{m+5}=0$. Clearly, $x^2u^3=0$, $x^{m+1}u^2=0$, $x^{m+3}u=0$. Combining with \eqref{eq-um+1=0}, then
\[
\text{Tot}E_\infty ^{*,*}\cong \zz_2[x,u]/\langle x^{m+5}, u^{m+1}, x^2u^3, x^{m+1}u^2, x^{m+3}u\rangle.
\] 

Analysing the filtration of $H^*(X_G)$ as in \eqref{eq-fil2} and \eqref{eq-fil3} and choosing the particular $y$ as in \eqref{eq-yy'}, consider \eqref{eq-Einfty5.1.1}, we get the following relations: 
\begin{align*}
y^{m+1} & =\alpha_1xy^m+\alpha_2x^{m-1}y^2+\alpha _3x^my+\alpha _4x^{m+1},\\
x^2y^3 & =\beta _1x^3y^2+\beta _2x^4y+\beta_3 x^5, \\
x^{m+1}y^2 & =\gamma _1x^{m+2}y+\gamma _2x^{m+3}, \\
x^{m+3}y & =0
\end{align*}
for some $\alpha_i~(i=1,\dots ,4)$, $\beta_j ~(j=1,2,3)$, $\gamma_1, \gamma_2 \in \zz_2$.
So the graded commutative algebra $H^*(X_G)$ is  $\zz_2[x,y]/I$, where $I$ is the ideal given by
\begin{align*}
I=\langle  & x^{m+5},y^{m+1}+\alpha_1xy^m+\alpha_2x^{m-1}y^2+\alpha _3x^my+\alpha _4x^{m+1},\\
& x^2y^3+\beta _1x^3y^2+\beta _2x^4y+\beta_3 x^5, x^{m+1}y^2+\gamma _1x^{m+2}y+\gamma _2x^{m+3}, x^{m+3}y\rangle,
\end{align*}
where  $m\geqslant 3$. This gives possibility (6) of \Cref{thm-Z2faRPmS4}.

\textbf{Subcase(\romannumeral 5.1.2)}
If $d_{m+3}:E_{m+3}^{0,m+3}\to E_{m+3}^{m+3,1}$ is trivial and $d_{m+4}:E_{m+4}^{0,m+3}\to E_{m+4}^{m+4,0}$ is nontrivial, then 
\begin{align}\label{eq-case5-dm+3}
	& d_{m+3}=0: E_{m+3}^{k,l}\to E_{m+3}^{k+m+3,l-m-2}, \text{for any } k, l, \\
	& d_{m+4}(1\otimes a^{m-1}b)=t^{m+4}\otimes 1, \nonumber\\
	& d_{m+4}(1\otimes a^mb)=t^{m+4}\otimes a. \nonumber
\end{align}
Furthermore, we obtain that
\begin{equation}\label{eq-case5-dm+4}
	d_{m+4}:E_{m+4}^{k,l}\to E_{m+4}^{k+m+4,l-m-3}
\end{equation}
is an isomorphism for all $k$ and $l=m+3,m+4$ and a trivial homomorphism otherwise. Consequently, by \eqref{eq-case5.1-Em+3}, \eqref{eq-case5-dm+3} and \eqref{eq-case5-dm+4}, we have
\begin{align*}
	E_{m+5}^{k,l} & =
	\begin{cases}
		\zz_2 , & m+1\leqslant k\leqslant m+3;l=0,1,\\
		\zz_2 , & 2\leqslant k\leqslant m;l=0,1,2,\\
		\zz_2 , & k=0,1;0\leqslant l\leqslant m,\\
		0 , & \text{otherwise}.
	\end{cases}
\end{align*} 
Note that
\(
d_r:E_r^{k,l}\to E_r^{k+r,l-r+1}
\)
is zero for all $r\geqslant m+5$ as $E_r^{k+r,l-r+1}=0$, so 
\[
E_{m+5}^{*,*}=E_\infty ^{*,*}. 
\]
We observe that the cohomology groups $H^j(X_G)$ are the the same as \eqref{eq-HjXG}. 

As $E_\infty ^{m+4,0}=0$, by \eqref{eq-x=pi*t}, we have $x^{m+4}=0$. Clearly, $x^2u^3=0$, $x^{m+1}u^2=0$. Combining with \eqref{eq-um+1=0}, then
\[
\text{Tot}E_\infty ^{*,*}\cong \zz_2[x,u]/\langle x^{m+4}, u^{m+1}, x^2u^3, x^{m+1}u^2\rangle.
\] 
The graded commutative algebra $H^*(X_G)$ is  $\zz_2[x,y]/I$, where $I$ is the ideal given by
\begin{align*}
I=\langle & x^{m+4},y^{m+1}+\alpha_1xy^m+\alpha_2x^{m-1}y^2+\alpha _3x^my+\alpha _4x^{m+1},\\
& x^2y^3+\beta _1x^3y^2+\beta _2x^4y+\beta_3 x^5, x^{m+1}y^2+\gamma _1x^{m+2}y+\gamma _2x^{m+3}\rangle,
\end{align*}
where $m\geqslant 3$ and $\alpha _i~(i=1,\dots ,4)$, $\beta_j ~(j=1,2,3)$, $\gamma_1, \gamma_2\in \zz_2$. This gives possibility (7) of \Cref{thm-Z2faRPmS4}.

\textbf{Subcase(\romannumeral 5.2)}
If $d_{m+1}:E_{m+1}^{0,m+2}\to E_{m+1}^{m+1,2}$ is trivial and $d_{m+2}:E_{m+2}^{0,m+2}\to E_{m+2}^{m+2,1}$ is nontrivial, then 
\begin{align}\label{eq-case5-dm+1}
	& d_{m+1}=0: E_{m+1}^{k,l}\to E_{m+1}^{k+m+1,l-m}, \text{for any } k, l, \\
	& d_{m+2}(1\otimes a^{m-2}b)=t^{m+2}\otimes a, \nonumber\\
	& d_{m+2}(1\otimes a^{m-1}b)=t^{m+2}\otimes a^2. \nonumber
\end{align}
Furthermore, we obtain that
\begin{equation}\label{eq-case5-dm+2}
	d_{m+2}:E_{m+2}^{k,l}\to E_{m+2}^{k+m+2,l-m-1}
\end{equation}
is an isomorphism for all $k$ and $l=m+2,m+3$ and a trivial homomorphism otherwise. Consequently, by \eqref{eq-case5-Em+1}, \eqref{eq-case5-dm+1} and \eqref{eq-case5-dm+2}, we have
\begin{align*}
	E_{m+3}^{k,l} & =
	\begin{cases}
		\zz_2 , & k\geqslant m+2;l=0,m+4,\\
		\zz_2 , & 2\leqslant k\leqslant m+1;l=0,1,2,m+4,\\
		\zz_2 , & k=0,1;0\leqslant l\leqslant m, l=m+4,\\
		0 , & \text{otherwise}.
	\end{cases}
\end{align*} 
The differential $d_r:E_r^{0,m+4}\to E_r^{r,m+5-r}$~($r\geqslant m+3$) can only be nontrivial when $r=m+5$. If $d_{m+5}:E_{m+5}^{0,m+4}\to E_{m+5}^{m+5,0}$ is trivial, then $d_{m+5}=0:E_{m+5}^{k,m+4}\to E_{m+5}^{k+m+5,0}$ for any $k$.  Thus $E_{m+3}^{*,*}=E_\infty ^{*,*}$, the bottom line ($l=0$) and the top line ($l=m+4$) of the spectral sequence survive to $E_\infty $, which contradicts \Cref{Bredonthm15}. Therefore, $d_{m+5}:E_{m+5}^{0,m+4}\to E_{m+5}^{m+5,0}$ is nontrivial. Then 
\(
d_{m+5}:E_{m+5}^{k,l}\to E_{m+5}^{k+m+5,l-m-4}
\) 
is an isomorphism for all $k$ and $l=m+4$ and a trivial homomorphism otherwise. Consequently, 
\begin{align*}
	E_{m+6}^{k,l} & =
	\begin{cases}
		\zz_2 , & m+2\leqslant k \leqslant m+4;l=0,\\
		\zz_2 , & 2\leqslant k\leqslant m+1;l=0,1,2,\\
		\zz_2 , & k=0,1;0\leqslant l\leqslant m,\\
		0 , & \text{otherwise}.
	\end{cases}
\end{align*}
Note that
\(
d_r:E_r^{k,l}\to E_r^{k+r,l-r+1}
\)
is zero for all $r\geqslant m+6$ as $E_r^{k+r,l-r+1}=0$, so 
\[
E_{m+6}^{*,*}=E_\infty ^{*,*}. 
\]
We observe that the cohomology groups $H^j(X_G)$ are the same as \eqref{eq-HjXG}.

As $E_\infty ^{m+5,0}=0$, by \eqref{eq-x=pi*t}, we have $x^{m+5}=0$. Clearly, $x^2u^3=0$, $x^{m+2}u=0$. Combining with \eqref{eq-um+1=0}, then
\[
\text{Tot}E_\infty ^{*,*}\cong \zz_2[x,u]/\langle x^{m+5}, u^{m+1}, x^2u^3, x^{m+2}u\rangle.
\] 
The graded commutative algebra  $H^*(X_G)$ is $\zz_2[x,y]/I$, where $I$ is the ideal given by
\begin{align*}
I=\langle & x^{m+5},y^{m+1}+\alpha_1xy^m+\alpha_2x^{m-1}y^2+\alpha _3x^my+\alpha _4x^{m+1},\\
& x^2y^3+\beta _1x^3y^2+\beta _2x^4y+\beta_3 x^5, x^{m+2}y\rangle,
\end{align*}
where $m\geqslant 3$ and $\alpha_i~(i=1,\dots ,4)$, $\beta_j~(j=1,2,3)\in \zz_2$. This gives possibility (8) of \Cref{thm-Z2faRPmS4}.

\textbf{Subcase(\romannumeral 5.3)}
If $d_r:E_r^{0,m+2}\to E_r^{r,m+3-r}$ is trivial for $r=m+1,m+2$ and $d_{m+3}:E_{m+3}^{0,m+2}\to E_{m+3}^{m+3,0}$ is nontrivial, then 
\begin{align}\label{eq-case5-dm+1m+2}
	& d_{r}=0: E_{r}^{k,l}\to E_{r}^{k+r,l+1-r}, \text{for any } k, l\text{ and } r=m+1, m+2\\
	& d_{m+3}(1\otimes a^{j}b)=t^{m+3}\otimes a^{j-m+2}, j=m-2, m-1, m. \nonumber
\end{align}
Furthermore, we obtain that
\begin{equation}\label{eq-case5.3-dm+3}
	d_{m+3}:E_{m+3}^{k,l}\to E_{m+3}^{k+m+3,l-m-2}
\end{equation}
is an isomorphism for all $k$ and $l=m+2,m+3,m+4$ and a trivial homomorphism otherwise. Consequently, by \eqref{eq-case5-Em+1}, \eqref{eq-case5-dm+1m+2} and \eqref{eq-case5.3-dm+3}, we have
\begin{align*}
	E_{m+4}^{k,l} & =
	\begin{cases}
		\zz_2 , & 2\leqslant k\leqslant m+2;l=0,1,2,\\
		\zz_2 , & k=0,1;0\leqslant l\leqslant m,\\
		0 , & \text{otherwise}.
	\end{cases}
\end{align*} 
Note that
\(
d_r:E_r^{k,l}\to E_r^{k+r,l-r+1}
\)
is zero for all $r\geqslant m+4$ as $E_r^{k+r,l-r+1}=0$, so 
\[
E_{m+4}^{*,*}=E_\infty ^{*,*}. 
\]
We observe that the cohomology groups $H^j(X_G)$ are the same as \eqref{eq-HjXG}.

As $E_\infty ^{m+3,0}=0$, by \eqref{eq-x=pi*t}, we have $x^{m+3}=0$. Clearly, $x^2u^3=0$. Combining with \eqref{eq-um+1=0}, then
\[
\text{Tot}E_\infty ^{*,*}\cong \zz_2[x,u]/\langle x^{m+3}, u^{m+1}, x^2u^3\rangle.
\] 
The graded commutative algebra $H^*(X_G)$ is  $\zz_2[x,y]/I$, where $I$ is the ideal given by
\[
I=\langle x^{m+3},y^{m+1}+\alpha_1xy^m+\alpha_2x^{m-1}y^2+\alpha _3x^my+\alpha _4x^{m+1},x^2y^3+\beta _1x^3y^2+\beta _2x^4y+\beta_3 x^5\rangle,
\] 
where $m\geqslant 3$ and $\alpha _i~(i=1,\dots ,4)$, $\beta_j~(j=1,2,3)\in \zz_2$. This gives possibility (9) of \Cref{thm-Z2faRPmS4}. \qed

\subsection{Proof of \Cref{thm-Z2faCPmS4}}

Let $G=\zz_2$ act freely on $X\sim _2 \cp^m \times S^4$.  For $m\geqslant 2$, we have
\begin{align*}
	H^l(X) & =
	\begin{cases}
		\zz_2, & l=0,2,2m+2,2m+4,\\
		(\zz_2)^2, & l=4,6,\dots ,2m,\\
		0, & \text{otherwise}.
	\end{cases}
\end{align*}
For $m=1$, we have
\begin{align*}
	H^l(X) & =
	\begin{cases}
		\zz_2, & l=0,2,4,6,\\
		0, & \text{otherwise}.
	\end{cases}
\end{align*}
Note that $E_2^{k,l}=H^k(B_G)\otimes H^l(X)=0$ for $l$ odd. This gives $d_r=0$ for $r$ even. Let $a\in H^2(X)$ and $b\in H^4(X)$ be generators of the cohomology algebra of $H^*(X)$, satisfying $a^{m+1}=0$ and $b^2=0$. As in the proof of \Cref{thm-Z2faRPmS4}, it is clear that $t\otimes 1\in E_2^{1,0}$ is a permanent cocycle and survives to a nontrivial element $x\in E_{\infty}^{1,0}$, i.e., 
\begin{equation}\label{eq-P21}
	x=\pi ^*(t)\in E_\infty ^{1,0}\subset H^1(X_G). 
\end{equation}

Since $\zz_2$ acts freely on $X$, by \Cref{Bredonthm15}, the spectral sequence does not collapse. Otherwise, we get $H^i(X/G)\ne 0$ for infinitely many values of $i>2m+4$. This implies that some differential $d_r:E_r^{k,l}\to E_r^{k+r,l-r+1}$ must be nontrivial. Note that $E_2^{*,*}$ is generated by $t\otimes 1\in E_2^{1,0}$, $1\otimes a\in E_2^{0,2}$ and $1\otimes b\in E_2^{0,4}$. There can only be nontrivial differentials $d_r$ on the generators when $r=3,5$. It follows immediately that there are three possibilities for nontrivial differentials:

\begin{enumerate}[(i)]\itemindent=2em
	\item $d_3(1\otimes a)\ne 0$.
	\item $d_3(1\otimes a)=0$, $d_3(1\otimes b)=0$ and $d_5(1\otimes b)\ne 0$.
	\item $d_3(1\otimes a)=0$ and $d_3(1\otimes b)\ne 0$.
\end{enumerate}

\textbf{Case(\romannumeral 1)} 
$d_3(1\otimes a)=t^3\otimes 1\ne 0$. 

If $m$ is even, then $a^{m+1}=0$ gives $0=d_3((1\otimes a^m)(1\otimes a))=t^3 \otimes a^m$, a contradiction. Hence $m$ must be odd. There are two possible subcases: either $d_3(1\otimes b)=t^3\otimes a\ne 0$ or $d_3(1\otimes b)=0$.

Firstly, let's consider $d_3(1\otimes b)=t^3\otimes a\ne 0$. Note that by the derivation property of the differential we have
\begin{align*}
	\begin{cases}
		d_3(1\otimes a^j)=j(t^3\otimes a^{j-1}), & 1\leqslant j\leqslant m, \\
		d_3(1\otimes a^jb)=j(t^3\otimes a^{j-1}b)+t^3\otimes a^{j+1}, & 0\leqslant j\leqslant m-1, \\
		d_3(1\otimes a^mb)=t^3\otimes a^{m-1}b.
	\end{cases}
\end{align*}
Note that 
\begin{align*}
	d_3(1\otimes ab) & =
	\begin{cases}
		t^3\otimes b+t^3\otimes a^2, & m>1,\\
		t^3\otimes b, & m=1.
	\end{cases}\\
	d_3d_3(1\otimes ab) & =
	\begin{cases}
		d_3(t^3\otimes b+t^3\otimes a^2), & m>1,\\
		d_3(t^3\otimes b), & m=1.
	\end{cases}\\
	& =t^6\otimes a\ne 0
\end{align*} 
This contradicts $d_3d_3=0$, thus $d_3(1\otimes b)=0$. 

By the derivation property of the differential we have
\begin{align*}
	\begin{cases}
		d_3(1\otimes a^j)=j(t^3\otimes a^{j-1}), & 1\leqslant j\leqslant m, \\
		d_3(1\otimes a^jb)=j(t^3\otimes a^{j-1}b), & 0\leqslant j\leqslant m.
	\end{cases}
\end{align*}
The $E_3$-term and $d_3$-differentials look like \Cref{case1-d3}. 
\begin{figure}[h]
\centering
\DeclareSseqGroup\tower{}{\class(0,0)\foreach \y in {1,...,6}{\class(0,\y) \structline }}
\DeclareSseqGroup\towerfive{}{\class(0,0)\foreach \y in {1,...,5}{\class(0,\y) \structline }}
\DeclareSseqGroup\towerfour{}{\class(0,0)\foreach \y in {1,...,4}{\class(0,\y) \structline }}
\DeclareSseqGroup\towerthree{}{\class(0,0)\foreach \y in {1,...,3}{\class(0,\y) \structline }}

\begin{sseqpage}[classes = fill, class labels = { left = 0.2em }, x tick step = 2, xscale = 1.3, yscale = 0.8, no ticks,class placement transform = {scale = 2}]
	\class[white](0,0)\class[white](7,0)
	\towerthree[struct lines = magenta, classes = black](2,0) \class[black](2,4) \class[black](2,5) \class[white](2,6) \class[white](2,7)
	\draw[dashed, magenta](1.74,3)--(1.74,4);\draw[magenta](1.74,4)--(1.74,5);
	\draw[magenta](2.26,2)--(2.26,3);\draw[dashed, magenta](2.26,3)--(2.26,4); \draw[magenta](2.26,4)--(2.26,5); \draw[magenta](2.26,5)--(2.26,6);\draw[magenta](2.26,6)--(2.26,7);
	\draw[cyan](2,0)--(2.26,2);\draw[cyan](2,1)--(2.26,3);\draw[cyan](2,4)--(2.26,6);\draw[cyan](2,5)--(2.26,7);
	\foreach \foreach \y in {2,3}\draw[cyan](2,\y)--(1.93,0.73+\y);
	\class[white](2,0) \class[white](2,1) \class[black](2,2) \class[black](2,3) \class[black](2,4) \class[black](2,5) \class[black](2,6) \class[black](2,7)
	
	\towerthree[struct lines = magenta, classes = black](4,0) \class[black](4,4) \class[black](4,5) \class[white](4,6) \class[white](4,7)
	\draw[dashed, magenta](3.74,3)--(3.74,4);\draw[magenta](3.74,4)--(3.74,5);
	\draw[magenta](4.26,2)--(4.26,3);\draw[dashed, magenta](4.26,3)--(4.26,4); \draw[magenta](4.26,4)--(4.26,5); \draw[magenta](4.26,5)--(4.26,6);\draw[magenta](4.26,6)--(4.26,7);
	\draw[cyan](4,0)--(4.26,2);\draw[cyan](4,1)--(4.26,3);\draw[cyan](4,4)--(4.26,6);\draw[cyan](4,5)--(4.26,7);
	\foreach \foreach \y in {2,3,4}\draw[cyan](4,\y)--(3.93,0.73+\y);
	\class[white](4,0) \class[white](4,1) \class[black](4,2) \class[black](4,3) \class[black](4,4) \class[black](4,5) \class[black](4,6) \class[black](4,7)
	
	\towerthree[struct lines = magenta, classes = black](6,0) \class[black](6,4) \class[black](6,5) \class[white](6,6) \class[white](6,7)
	\draw[dashed, magenta](5.74,3)--(5.74,4);\draw[magenta](5.74,4)--(5.74,5);
	\draw[magenta](6.26,2)--(6.26,3);\draw[dashed, magenta](6.26,3)--(6.26,4); \draw[magenta](6.26,4)--(6.26,5); \draw[magenta](6.26,5)--(6.26,6);\draw[magenta](6.26,6)--(6.26,7);
	\draw[cyan](6,0)--(6.26,2);\draw[cyan](6,1)--(6.26,3);\draw[cyan](6,4)--(6.26,6);\draw[cyan](6,5)--(6.26,7);
	\foreach \foreach \y in {2,3,4}\draw[cyan](6,\y)--(5.93,0.73+\y);
	\class[white](6,0) \class[white](6,1) \class[black](6,2) \class[black](6,3) \class[black](6,4) \class[black](6,5) \class[black](6,6) \class[black](6,7)
	
	\node[background] at (0,-1) {\cdots};
	\node[background] at (2,-1) {k-3\geqslant 0}; \node[background] at (4,-1) {k}; \node[background] at (6,-1) {k+3};
	\node[background] at (-1,0) {0};\node[background] at (-1,1) {2};
	\node[background] at (-1,2) {4};\node[background] at (-1,3) {6};
	\node[background] at (-1,3.6) {\vdots};\node[background] at (-1.3,4) {2m-2}; \node[background] at (-1.05,5) {2m};
	\node[background] at (-1.3,6) {2m+2};\node[background] at (-1.3,7) {2m+4};
	
	\draw[blue,->](2,1)--(4,0);\draw[blue,->](2,3)--(4,2);\draw[blue,->](2,5)--(4,4);
	\draw[blue,->](2,3,2)--(4,2,2);\draw[blue,->](2,5,2)--(4,4,2);\draw[blue,->](2,7,2)--(4,6,2);
	
	\draw[blue,->](4,1)--(6,0);\draw[blue,->](4,3)--(6,2);\draw[blue,->](4,5)--(6,4);
	\draw[blue,->](4,3,2)--(6,2,2);\draw[blue,->](4,5,2)--(6,4,2);\draw[blue,->](4,7,2)--(6,6,2);
	
	\foreach \foreach \y in {0,1,2,3,4,5,6,7}\draw[lightgray!50, thin](0,\y)--(8,\y);
	
	\node[background] at (1.28,0.05) {t^{k-3}};\node[background] at (1.20,1.05) {t^{k-3}a};\node[background] at (1.10,5.05) {t^{k-3}a^m};
	\node[background] at (3.43,0.05) {t^{k}};\node[background] at (3.35,1.05) {t^{k}a};\node[background] at (3.25,5.05) {t^{k}a^m};
	\node[background] at (5.28,0.05) {t^{k+3}};\node[background] at (5.20,1.05) {t^{k+3}a};\node[background] at (5.15,5.05) {t^{k+3}a^m};
	
	\node[background] at (2.8,2.05) {t^{k-3}b};\node[background] at (3.0,7.05) {t^{k-3}a^mb};
	\node[background] at (4.65,2.05) {t^{k}b};\node[background] at (4.7,3.05) {t^{k}ab};\node[background] at (4.85,7.05) {t^{k}a^mb};
	\node[background] at (6.8,3.05) {t^{k+3}b};\node[background] at (7.0,7.05) {t^{k+3}a^mb};
\end{sseqpage}
\caption{$E_3$-term and $d_3$-differentials in \textbf{Case(\romannumeral 1)}}\label{case1-d3}
\end{figure}
Then
\begin{align*}
	E_4^{k,l} & =
	\begin{cases}
		\zz_2 , & 0\leqslant k\leqslant2; l=0,2m+2,\\
		(\zz_2)^2, & 0\leqslant k\leqslant2;l=4,8,12,\dots ,2m-2,\\
		0 , & \text{otherwise}.
	\end{cases}
\end{align*}  
Note that $d_r:E_r^{k,l}\to E_r^{k+r,l-r+1}$ is zero for all $r\geqslant 4$ as $E_r^{k+r,l-r+1}=0$, so 
\[
E_4^{*,*}=E_\infty ^{*,*}. 
\]

Since $H^*(X_G)\cong \text{Tot}E_\infty ^{*,*}$, the additive structure of $H^*(X_G)$ is given by 
\begin{align*}
	H^j(X_G) & =
	\begin{cases}
		\zz_2 , & 0\leqslant j\leqslant 2 \text{~or~} 2m+2\leqslant j\leqslant 2m+4,\\
		(\zz_2)^2 , & 4\leqslant j\leqslant 2m \text{~and~} j\ne 7,11,15, \dots ,2m-3,\\
		0 , & \text{otherwise}.
	\end{cases}
\end{align*} 

As $E_\infty ^{3,0}=0$, by \eqref{eq-P21}, we have $x^3=0$. Notice that, the elements $1\otimes a^2\in E_2^{0,4}$ and $1\otimes b\in E_2^{0,4}$ are permanent cocycles and are not hit by any $d_r$-coboundaries. Hence, they determine nontrivial elements $u\in E_\infty ^{0,4}$ and $v\in E_\infty ^{0,4}$, respectively. We have $u^{\frac{m+1}{2} }=0$ as $a^{m+1}=0$, and $v^2=0$ as $b^2=0$. Thus
\[
\text{Tot}E_\infty ^{*,*}\cong \zz_2[x,u,v]/\langle x^2,u^{\frac{m+1}{2}},v^2 \rangle,
\] 
where $\deg x=1$, $\deg u=4$, $\deg v=4$.

By the edge homomorphism, let $y\in H^4(X_G)$ and $z\in H^4(X_G)$ be such that $i^*(y)=a^2$ and $i^*(z)=b$, respectively. Notice that $y^{\frac{m+1}{2}}\in H^{2m+2}(X_G)=E_{\infty}^{0,2m+2}$ is represented by $a^{m+1}\in E_{2}^{0,2m+2}$ and $z^2\in H^8(X_G)=E_{\infty}^{0,8}$ is represented by $b^2\in E_{2}^{0,8}$. Since the edge homomorphism is an isomorphism in degrees 8 and $2m+2$, we have the following relations:
\[
y^{\frac{m+1}{2}}=0,~z^2=0.
\]
Therefore,
\[
H^*(X_G)=\zz_2[x,y,z]/\langle x^3,y^{\frac{m+1}{2}},z^2\rangle,
\] 
where $\deg x=1$, $\deg y=4$, $\deg z=4$ and $m$ is odd. This gives possibility (1) of \Cref{thm-Z2faCPmS4}.

\textbf{Case(\romannumeral 2)}
$d_3(1\otimes a)=0$, $d_3(1\otimes b)=0$ and $d_5(1\otimes b)=t^5\otimes 1\ne 0$.

This case implies that $d_3=0$. We have
\begin{align*}
	\begin{cases}
		d_5(1\otimes a^j)=0, & 1\leqslant j\leqslant m, \\
		d_5(1\otimes a^jb)=t^5\otimes a^j, & 0\leqslant j\leqslant m.
	\end{cases}
\end{align*}
and
\[
\xymatrix@R=2mm{
	E_5^{k-5,l+4}\ar[r]^-{d_5} & E_5^{k,l}\ar[r]^-{d_5} & E_5^{k+5,l-4},\\
	t^{k-5}\otimes a^{\frac{l}{2}}b\ar@{|->}[r]^-{d_5} & t^k\otimes a^{\frac{l}{2}}\ar@{|->}[r]^-{d_5} & 0,  \\
	t^{k-5}\otimes a^{\frac{l}{2}+4}\ar@{|->}[r]^-{d_5} & 0,~t^k\otimes a^{\frac{l}{2}-4}b\ar@{|->}[r]^-{d_5} & t^{k+5}\otimes a^{\frac{l}{2}-4}.}
\]
So
\begin{align*}
	E_6^{k,l} & =
	\begin{cases}
		\zz_2 , & 0\leqslant k\leqslant 4;l=0,2,\dots ,2m,\\
		0 , & \text{otherwise}.
	\end{cases}
\end{align*} 
Note that
\(
d_r:E_r^{k,l}\to E_r^{k+r,l-r+1}
\)
is zero for all $r\geqslant 6$ as $E_r^{k+r,l-r+1}=0$, so 
\[
E_6^{*,*}=E_\infty ^{*,*}. 
\]
The additive structure of $H^*(X_G)$ is given by 
\begin{align}\label{eq-cpms4caseii}
	H^j(X_G) & =
	\begin{cases}
		\zz_2 , & j=0,1,2m+3,2m+4,\\
		(\zz_2)^2, & j=2,2m+2 \text{ or }j=3,5,\dots ,2m+1,\\
		(\zz_2)^3, & j=4,6,\dots ,2m,\\
		0 , & \text{otherwise}.
	\end{cases}
\end{align} 

Notice that, the element $1\otimes a\in E_2^{0,2}$ is a permanent cocycle and is not a $d_r$-coboundary. Hence, it determines a nontrivial element $u\in E_\infty ^{0,2}$. As we have remarked, $a^{m+1}=0$, so 
\begin{equation}\label{eq-cpm-um+1=0}
	u^{m+1}=0.
\end{equation}
As $E_\infty ^{5,0}=0$,  by \eqref{eq-P21}, we have $x^5=0$. Thus
\[
\text{Tot}E_\infty ^{*,*}\cong \zz_2[x,u]/\langle x^5,u^{m+1}\rangle,
\] 
where $\deg x=1$, $\deg u=2$. 

Now, choose $y\in H^2(X_G)$ such that $i^*(y)=a$. By considering the filtration on $H^{2m+2}(X_G)$, 
\begin{equation}\label{eq-fil4}
	0=F_{2m+2}^{2m+2}=\cdots =F_{5}^{2m+2}\subset F_{4}^{2m+2}=F_{3}^{2m+2} \subset F_{2}^{2m+2}=F_{1}^{2m+2}= F_{0}^{2m+2}=H^{2m+2}(X_G),
\end{equation}
\vskip -8mm

\hskip 4.cm$\underbrace{\hskip 11mm}_{E_{\infty}^{4,2m-2}}$~
\hskip 21mm$\underbrace{\hskip 11mm}_{E_{\infty}^{2,2m}}$~
\hskip 3mm

\noindent we get the following relation: 
\[
y^{m+1}=\alpha_1x^2y^m+\alpha _2x^4y^{m-1},
\] 
where $\alpha_1, \alpha_2\in \zz_2$. Therefore, 
\[
H^*(X_G)=\zz_2[x,y]/\langle x^5,y^{m+1}+\alpha_1x^2y^m+\alpha _2x^4y^{m-1}\rangle,
\] 
where $\deg x=1$, $\deg y=2$. This gives possibility (2) of \Cref{thm-Z2faCPmS4}.

{\bf In the remaining \textbf{Case(\romannumeral 3)} there will be classes $u\in E_\infty ^{0,2}$, $y\in H^2(X_G)$ defined as above and the relation \eqref{eq-cpm-um+1=0} will be satisfied}. 

\textbf{Case(\romannumeral 3)}
$d_3(1\otimes a)=0$ and $d_3(1\otimes b)\ne 0$. 

Clearly, $d_3(1\otimes b)=t^3\otimes a$. So we have
\begin{align*}
	\begin{cases}
		d_3(1\otimes a^j)=0, & 1\leqslant j\leqslant m, \\
		d_3(1\otimes a^jb)=t^3\otimes a^{j+1}, & 0\leqslant j\leqslant m-1, \\
		d_3(1\otimes a^mb)=0.
	\end{cases}
\end{align*}
The $E_3$-term and $d_3$-differentials look like \Cref{case3-d3}. 
\begin{figure}[h]
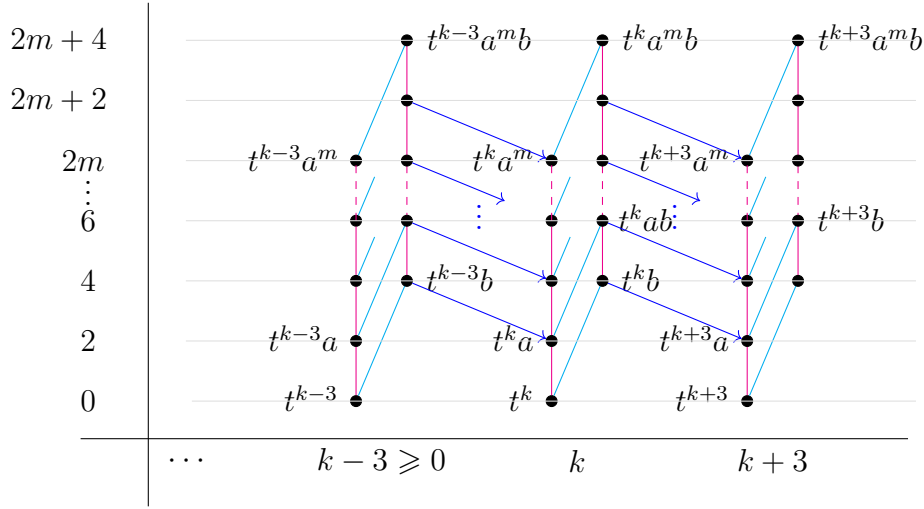

\centering
\DeclareSseqGroup\tower{}{\class(0,0)\foreach \y in {1,...,6}{\class(0,\y) \structline }}
\DeclareSseqGroup\towerfive{}{\class(0,0)\foreach \y in {1,...,5}{\class(0,\y) \structline }}
\DeclareSseqGroup\towerfour{}{\class(0,0)\foreach \y in {1,...,4}{\class(0,\y) \structline }}
\DeclareSseqGroup\towerthree{}{\class(0,0)\foreach \y in {1,...,3}{\class(0,\y) \structline }}
\begin{sseqpage}[classes = fill, class labels = { left = 0.2em }, x tick step = 2, xscale = 1.3, yscale = 0.8, no ticks,class placement transform = {scale = 2}]
	\class[white](0,0)\class[white](7,0)
	\towerthree[struct lines = magenta, classes = black](2,0) \class[black](2,4) \class[white](2,5) \class[white](2,6)
	\draw[dashed, magenta](1.74,3)--(1.74,4);
	\draw[magenta](2.26,2)--(2.26,3);\draw[dashed, magenta](2.26,3)--(2.26,4);\draw[magenta](2.26,4)--(2.26,5);\draw[magenta](2.26,5)--(2.26,6);
	\draw[cyan](2,0)--(2.26,2);\draw[cyan](2,1)--(2.26,3);\draw[cyan](2,4)--(2.26,6);
	\foreach \foreach \y in {2,3}\draw[cyan](2,\y)--(1.93,0.73+\y);
	\class[white](2,0) \class[white](2,1) \class[black](2,2) \class[black](2,3) \class[black](2,4) \class[black](2,5) \class[black](2,6)
	
	\towerthree[struct lines = magenta, classes = black](4,0) \class[black](4,4) \class[white](4,5) \class[white](4,6)
	\draw[dashed, magenta](3.74,3)--(3.74,4);
	\draw[magenta](4.26,2)--(4.26,3);\draw[dashed, magenta](4.26,3)--(4.26,4);\draw[magenta](4.26,4)--(4.26,5);\draw[magenta](4.26,5)--(4.26,6);
	\draw[cyan](4,0)--(4.26,2);\draw[cyan](4,1)--(4.26,3);\draw[cyan](4,4)--(4.26,6);
	\foreach \foreach \y in {2,3}\draw[cyan](4,\y)--(3.93,0.73+\y);
	\class[white](4,0) \class[white](4,1) \class[black](4,2) \class[black](4,3) \class[black](4,4) \class[black](4,5) \class[black](4,6)
	
	\towerthree[struct lines = magenta, classes = black](6,0) \class[black](6,4) \class[white](6,5) \class[white](6,6)
	\draw[dashed, magenta](5.74,3)--(5.74,4);
	\draw[magenta](6.26,2)--(6.26,3);\draw[dashed, magenta](6.26,3)--(6.26,4);\draw[magenta](6.26,4)--(6.26,5);\draw[magenta](6.26,5)--(6.26,6);
	\draw[cyan](6,0)--(6.26,2);\draw[cyan](6,1)--(6.26,3);\draw[cyan](6,4)--(6.26,6);
	\foreach \foreach \y in {2,3}\draw[cyan](6,\y)--(5.93,0.73+\y);
	\class[white](6,0) \class[white](6,1) \class[black](6,2) \class[black](6,3) \class[black](6,4) \class[black](6,5) \class[black](6,6)
	
	\node[background] at (0,-1) {\cdots};
	\node[background] at (2,-1) {k-3\geqslant 0}; \node[background] at (4,-1) {k}; \node[background] at (6,-1) {k+3};
	\node[background] at (-1,0) {0};\node[background] at (-1,1) {2};
	\node[background] at (-1,2) {4};\node[background] at (-1,3) {6};
	\node[background] at (-1,3.6) {\vdots};\node[background] at (-1.05,4) {2m};
	\node[background] at (-1.3,5) {2m+2};\node[background] at (-1.3,6) {2m+4};
	
	\draw[blue,->](2,2,2)--(4,1);\draw[blue,->](2,3,2)--(4,2);\draw[blue,->](2,4,2)--(3.26,3.33);\draw[blue,->](2,5,2)--(4,4);
	
	\draw[blue,->](4,2,2)--(6,1);\draw[blue,->](4,3,2)--(6,2);\draw[blue,->](4,4,2)--(5.26,3.33);\draw[blue,->](4,5,2)--(6,4);
	\node[blue] at (3.0,3.2) {\vdots};\node[blue] at (5.0,3.2) {\vdots};
	
	\foreach \foreach \y in {0,1,2,3,4,5,6}\draw[lightgray!50, thin](0,\y)--(8,\y);
	
	\node[background] at (1.28,0.05) {t^{k-3}};\node[background] at (1.20,1.05) {t^{k-3}a};\node[background] at (1.10,4.05) {t^{k-3}a^m};
	\node[background] at (3.43,0.05) {t^{k}};\node[background] at (3.35,1.05) {t^{k}a};\node[background] at (3.25,4.05) {t^{k}a^m};
	\node[background] at (5.28,0.05) {t^{k+3}};\node[background] at (5.20,1.05) {t^{k+3}a};\node[background] at (5.10,4.05) {t^{k+3}a^m};
	
	\node[background] at (2.8,2.05) {t^{k-3}b};\node[background] at (3.0,6.05) {t^{k-3}a^mb};
	\node[background] at (4.65,2.05) {t^{k}b};\node[background] at (4.7,3.05) {t^{k}ab};\node[background] at (4.85,6.05) {t^{k}a^mb};
	\node[background] at (6.8,3.05) {t^{k+3}b};\node[background] at (7.0,6.05) {t^{k+3}a^mb};
\end{sseqpage}
\caption{$E_3$-term and $d_3$-differentials in \textbf{Case(\romannumeral 3)}}\label{case3-d3}
\end{figure}
Then
\begin{align*}
	E_5^{k,l} =E_4^{k,l} =
	\begin{cases}
		\zz_2 , & k\geqslant 3;l=0,2m+4,\\
		\zz_2 , & 0\leqslant k\leqslant 2;l=0,2,\dots ,2m\text{ or } l=2m+4,\\
		0 , & \text{otherwise}.
	\end{cases}
\end{align*} 
It is easy to see that 
\(
d_r:E_r^{k,l}\to E_r^{k+r,l-r+1}
\)
is zero for all $5\leqslant r\leqslant 2m+4$. Now if $d_{2m+5}:E_{2m+5}^{0,2m+4}\to E_{2m+5}^{2m+5,0}$ is trivial, then by the multiplicative properties of the spectral sequence, we have $E_{2m+5}^{*,*}=E_\infty ^{*,*}$. Therefore the bottom line ($l=0$) and the top line ($l=2m+4$) of the spectral sequence survive to $E_\infty $, which reduces to $H^i(X/G)\ne 0$ for all {$i>2m+4$}. This contradicts to \Cref{Bredonthm15}. Thus, $d_{2m+5}:E_{2m+5}^{0,2m+4}\to E_{2m+5}^{2m+5,0}$ must be nontrivial. 
It follows immediately that $d_{2m+5}:E_{2m+5}^{k,2m+4}\to E_{2m+5}^{k+2m+5,0}$ is an isomorphism for all $k$. So	
\begin{align*}
	E_{2m+6}^{k,l} & =
	\begin{cases}
		\zz_2 , & 3\leqslant k\leqslant 2m+4;l=0,\\
		\zz_2 , & 0\leqslant k\leqslant 2;l=0,2,\dots ,2m,\\
		0 , & \text{otherwise}.
	\end{cases}
\end{align*}
Note that
\(
d_r:E_r^{k,l}\to E_r^{k+r,l-r+1}
\)
is zero for all $r\geqslant 2m+6$ as $E_r^{k+r,l-r+1}=0$, so 
\[
E_{2m+6}^{*,*}=E_\infty ^{*,*}. 
\]
It follows that the cohomology groups $H^j(X_G)$ are the same \eqref{eq-cpms4caseii} as in {\bf  Case(\romannumeral 2)}. 

As $E_\infty ^{2m+5,0}=0$,  by \eqref{eq-P21}, we have $x^{2m+5}=0$. Clearly, $x^3u=0$. Combining with \eqref{eq-cpm-um+1=0}, then
\[
\text{Tot}E_\infty ^{*,*}\cong \zz_2[x,u]/\langle x^{2m+5},u^{m+1},x^3u\rangle,
\]  
Choose $y'\in H^2(X_G)$ such that $i^*(y')=a$ and let $y = y'+\beta x^2\in H^2(X_G)$, $\beta \in \zz_2$. As before, we conclude that the graded commutative algebra  $H^*(X_G)$ is $\zz_2[x,y]/I$, where $I$ is the ideal given by
\[
I=\langle x^{2m+5},y^{m+1}+\alpha_1x^2y^m+\alpha_2x^{2m+2},x^3y\rangle,
\] 
where $\alpha_1, \alpha_2 \in \zz_2$. This gives possibility (3) of \Cref{thm-Z2faCPmS4}. \qed

\subsection{Proof of \Cref{thm-Z2faHPmS4}}

Let $G=\zz_2$ act freely on $X\sim _2 \mathbb{H}P^m \times S^4$.  We observe that $m\geqslant 1$, 
\begin{align*}
	H^l(X) & =
	\begin{cases}
		\zz_2, & l=0,4m+4,\\
		(\zz_2)^2, & l=4,8,\dots ,4m,\\
		0, & \text{otherwise}.
	\end{cases}
\end{align*}
Note that $E_2^{k,l}=H^k(B_G)\otimes H^l(X)=0$ for $l \not\equiv 0 \pmod{4} $. This gives $d_r=0$ for $2 \leqslant r \leqslant 4$ and hence $E_2^{*,*}=E_5^{*,*}$. 
Let $a\in H^4(X)$ and $b\in H^4(X)$ be generators of the cohomology algebra of $H^*(X)$, satisfying $a^{m+1}=0$ and $b^2=0$. The element $t\otimes 1\in E_2^{1,0}$ is a permanent cocycle and survives to a nontrivial element $x\in E_{\infty}^{1,0}$, i.e.,
\begin{equation}\label{eq-P26}
	x=\pi ^*(t)\in E_\infty ^{1,0}\subset H^1(X_G) 
\end{equation} 

Since $\zz_2$ acts freely on $X$, by \Cref{Bredonthm15}, the spectral sequence does not collapse. It implies that some differential $d_r:E_r^{k,l}\to E_r^{k+r,l-r+1}$ must be nontrivial. Note that $E_2^{*,*}$ is generated by $t\otimes 1\in E_2^{1,0}$, $1\otimes a\in E_2^{0,4}$ and $1\otimes b\in E_2^{0,4}$. The first nontrivial differential $d_r$ occurs possibly only when $r=5$. 
It follows immediately that there are three possibilities for the nontrivial differential:

\begin{enumerate}[(i)]\itemindent=2em
	\item $d_5(1\otimes a)\ne 0$ and $d_5(1\otimes b)\ne 0$.
	\item $d_5(1\otimes a)\ne 0$ and $d_5(1\otimes b)=0$.
	\item $d_5(1\otimes a)=0$ and $d_5(1\otimes b)\ne 0$.
\end{enumerate}

\textbf{Case(\romannumeral 1)} 
$d_5(1\otimes a)=t^5\otimes 1\ne 0$ and $d_5(1\otimes b)=t^5\otimes 1\ne 0$.

Note that by the derivation property of the differential we have
\begin{align*}
	\begin{cases}
		d_5(1\otimes a^j)=j(t^5\otimes a^{j-1}), & 1\leqslant j\leqslant m, \\
		d_5(1\otimes a^jb)=t^5\otimes a^j+j(t^5\otimes a^{j-1}b), & 0\leqslant j\leqslant m.
	\end{cases}
\end{align*}
If $m$ is even, then $a^{m+1}=0$ gives $0=d_5((1\otimes a^m)(1\otimes a))=t^5 \otimes a^m$, a contradiction. Hence $m$ must be odd. The $E_5$-term and $d_5$-differentials look like \Cref{case1-d5-1}. 
\begin{figure}[h]
\centering
\DeclareSseqGroup\tower{}{\class(0,0)\foreach \y in {1,...,6}{\class(0,\y) \structline }}
\DeclareSseqGroup\towerfive{}{\class(0,0)\foreach \y in {1,...,5}{\class(0,\y) \structline }}
\DeclareSseqGroup\towerfour{}{\class(0,0)\foreach \y in {1,...,4}{\class(0,\y) \structline }}
\DeclareSseqGroup\towerthree{}{\class(0,0)\foreach \y in {1,...,3}{\class(0,\y) \structline }}
\begin{sseqpage}[classes = fill, class labels = { left = 0.2em }, x tick step = 2, xscale = 1.5, yscale = 0.8, no ticks,class placement transform = {scale = 2}]
	\class[white](0,0)\class[white](7,0)
	
	\class[black](2,0) \class[black](2,1) \class[black](2,2) \class[black](2,3) \class[black](2,4) \class[white](2,5)
	\draw[magenta](1.74,0)--(1.74,1);\draw[magenta](1.74,1)--(1.74,2); \draw[dashed, magenta](1.74,2)--(1.74,3);\draw[magenta](1.74,3)--(1.74,4);
	\draw[magenta](2.26,1)--(2.26,2);\draw[dashed, magenta](2.26,2)--(2.26,3);
	\draw[magenta](2.26,3)--(2.26,4);\draw[magenta](2.26,4)--(2.26,5);
	\draw[cyan](2,0)--(2.26,1);\draw[cyan](2,1)--(2.26,2);\draw[cyan](2,3)--(2.26,4);\draw[cyan](2,4)--(2.26,5);
	\draw[cyan](2,2)--(1.93,2.36);
	\class[white](2,0) \class[black](2,1) \class[black](2,2) \class[black](2,3) \class[black](2,4) \class[black](2,5)
	
	\class[black](4,0) \class[black](4,1) \class[black](4,2) \class[black](4,3) \class[black](4,4) \class[white](4,5)
	\draw[magenta](3.74,0)--(3.74,1);\draw[magenta](3.74,1)--(3.74,2); \draw[dashed, magenta](3.74,2)--(3.74,3);\draw[magenta](3.74,3)--(3.74,4);
	\draw[magenta](4.26,1)--(4.26,2);\draw[dashed, magenta](4.26,2)--(4.26,3);
	\draw[magenta](4.26,3)--(4.26,4);\draw[magenta](4.26,4)--(4.26,5);
	\draw[cyan](4,0)--(4.26,1);\draw[cyan](4,1)--(4.26,2);\draw[cyan](4,3)--(4.26,4);\draw[cyan](4,4)--(4.26,5);
	\draw[cyan](4,2)--(3.93,2.36);
	\class[white](4,0) \class[black](4,1) \class[black](4,2) \class[black](4,3) \class[black](4,4) \class[black](4,5)
	
	\class[black](6,0) \class[black](6,1) \class[black](6,2) \class[black](6,3) \class[black](6,4) \class[white](6,5)
	\draw[magenta](5.74,0)--(5.74,1);\draw[magenta](5.74,1)--(5.74,2); \draw[dashed, magenta](5.74,2)--(5.74,3);\draw[magenta](5.74,3)--(5.74,4);
	\draw[magenta](6.26,1)--(6.26,2);\draw[dashed, magenta](6.26,2)--(6.26,3);
	\draw[magenta](6.26,3)--(6.26,4);\draw[magenta](6.26,4)--(6.26,5);
	\draw[cyan](6,0)--(6.26,1);\draw[cyan](6,1)--(6.26,2);\draw[cyan](6,3)--(6.26,4);\draw[cyan](6,4)--(6.26,5);
	\draw[cyan](6,2)--(5.93,2.36);
	\class[white](6,0) \class[black](6,1) \class[black](6,2) \class[black](6,3) \class[black](6,4) \class[black](6,5)
	
	\node[background] at (0,-1) {\cdots};
	\node[background] at (2,-1) {k-5\geqslant 0}; \node[background] at (4,-1) {k}; \node[background] at (6,-1) {k+5};
	\node[background] at (-1,0) {0};\node[background] at (-1,1) {4};
	\node[background] at (-1,2) {8};
	\node[background] at (-1,2.6) {\vdots};\node[background] at (-1.3,3) {4m-4}; \node[background] at (-1.05,4) {4m};
	\node[background] at (-1.3,5) {4m+4};
	
	\draw[blue,->](2,1)--(4,0);\draw[blue,->](2,4)--(4,3);
	\draw[blue,->](2,1,2)--(4,0);\draw[blue,->](2,2,2)--(4,1);\draw[blue,->](2,3,2)--(3.26,2.32);\draw[blue,->](2,4,2)--(4,3);\draw[blue,->](2,5,2)--(4,4);
	\draw[blue,->](2,2,2)--(4,1,2);\draw[blue,->](2,3,2)--(3.62,2.32);\draw[blue,->](2,5,2)--(4,4,2);
	
	\draw[blue,->](4,1)--(6,0);\draw[blue,->](4,4)--(6,3);
	\draw[blue,->](4,1,2)--(6,0);\draw[blue,->](4,2,2)--(6,1);\draw[blue,->](4,3,2)--(5.26,2.32);\draw[blue,->](4,4,2)--(6,3);\draw[blue,->](4,5,2)--(6,4);
	\draw[blue,->](4,2,2)--(6,1,2);\draw[blue,->](4,3,2)--(5.62,2.32);\draw[blue,->](4,5,2)--(6,4,2);
	
	\foreach \foreach \y in {0,1,2,3,4,5}\draw[lightgray!50, thin](0,\y)--(8,\y);
	
	\node[background] at (1.33,0.05) {t^{k-5}};\node[background] at (1.25,1.05) {t^{k-5}a};\node[background] at (1.15,4.05) {t^{k-5}a^m};
	\node[background] at (3.48,0.05) {t^{k}};\node[background] at (3.40,1.05) {t^{k}a};\node[background] at (3.30,4.05) {t^{k}a^m};
	\node[background] at (5.33,0.05) {t^{k+5}};\node[background] at (5.25,1.05) {t^{k+5}a};\node[background] at (5.15,4.05) {t^{k+5}a^m};
	
	\node[background] at (2.75,1.05) {t^{k-5}b};\node[background] at (2.95,5.05) {t^{k-5}a^mb};
	\node[background] at (4.60,1.05) {t^{k}b};\node[background] at (4.67,2.05) {t^{k}ab};\node[background] at (4.8,5.05) {t^{k}a^mb};
	\node[background] at (6.75,1.05) {t^{k+5}b};\node[background] at (6.95,5.05) {t^{k+5}a^mb};
\end{sseqpage}
\caption{$E_5$-term and $d_5$-differentials in \textbf{Case(\romannumeral 1)}}\label{case1-d5-1}
\end{figure}
Then
\begin{align*}
	E_6^{k,l} & =
	\begin{cases}
		\zz_2 , & 0\leqslant k\leqslant 4; l=0,4,8,\dots ,4m,\\
		0 , & \text{otherwise}.
	\end{cases}
\end{align*} 
Note that $d_r:E_r^{k,l}\to E_r^{k+r,l-r+1}$ is zero for all $r\geqslant 6$ as $E_r^{k+r,l-r+1}=0$, so 
\[
E_6^{*,*}=E_\infty ^{*,*}. 
\]
Since $H^*(X_G)\cong \text{Tot}E_\infty ^{*,*}$, we have 
\begin{align*}
	H^j(X_G) & =
	\begin{cases}
		\zz_2 , & 0\leqslant j\leqslant 4m+4 \text{~and~} j \ne 4,8,\dots ,4m,\\
		(\zz_2)^2 , & j=4,8,\dots ,4m,\\
		0 , & \text{otherwise}.
	\end{cases}
\end{align*} 

As $E_\infty ^{5,0}=0$, by \eqref{eq-P26}, we have $x^5=0$. Notice that, the elements $1\otimes a^2\in E_2^{0,8}$ and $1\otimes (a+b)\in E_2^{0,4}$ are permanent cocycles and are not hit by any $d_r$-coboundaries. Hence, they determine nontrivial elements $u\in E_\infty ^{0,8}$ and $v\in E_\infty ^{0,4}$, respectively. We have $u^{\frac{m+1}{2} }=0$ as $a^{m+1}=0$, and $v^2+u=0$ as $b^2=0$. Thus
\[
\text{Tot}E_\infty ^{*,*}\cong \zz_2[x,u,v]/\langle x^5,u^{\frac{m+1}{2}},v^2+u \rangle,
\] 
where $\deg x=1$, $\deg u=8$, $\deg v=4$. 

Let $y\in H^8(X_G)$ and $z\in H^4(X_G)$ be such that $i^*(y)=a^2$ and $i^*(z)=a+b$, respectively. By considering the filtrations of $H^{4m+4}(X_G)$ and $H^8(X_G)$, we have the short exact sequence
\begin{equation}\label{eq-hpm-case1Hj=totalCPX}
	0\to E_{\infty}^{4,j-4}\to H^{j}(X_G)\xrightarrow[]{} E_{\infty}^{0,j}\to 0,~j=4m+4\text{~or~}8.
\end{equation}
By \eqref{eq-hpm-case1Hj=totalCPX}, we get the following relations: 
\begin{align*}
	y^{\frac{m+1}{2}} & =\beta x^4y^{\frac{m-1}{2}}z, ~\beta\in\zz_2,\\
	z^2+y & =\alpha x^4z, ~\alpha\in\zz_2.
\end{align*}
 Therefore, 
\[
H^*(X_G)=\zz_2[x,y,z]/\langle x^5,y^{\frac{m+1}{2}}+\beta x^4y^{\frac {m-1}{2}}z, z^2+\gamma y+\alpha x^4z\rangle,
\] 
where $\deg x=1$, $\deg y=8$, $\deg z=4$, $\alpha,\beta,\gamma \in \zz_2$ and $m$ is odd. Also $\gamma=1$ except when $m=1$. 

\textbf{Case(\romannumeral 2)} 
$d_5(1\otimes a)=t^5\otimes 1\ne 0$ and $d_5(1\otimes b)=0$. 

If $m$ is even, then $0=d_5(1\otimes a^{m+1})=t^5\otimes a^m$, a contradiction. So $m$ must be odd. Note that by the derivation property of the differential we have
\begin{align*}
	\begin{cases}
	d_5(1\otimes a^j)=j(t^5\otimes a^{j-1}), & 1\leqslant j\leqslant m, \\
	d_5(1\otimes a^jb)=j(t^5\otimes a^{j-1}b), & 0\leqslant j\leqslant m.
\end{cases}
\end{align*}
The $E_5$-term and $d_5$-differentials look like \Cref{case1-d5-2}. 
\begin{figure}[h]
\centering
\DeclareSseqGroup\tower{}{\class(0,0)\foreach \y in {1,...,6}{\class(0,\y) \structline }}
\DeclareSseqGroup\towerfive{}{\class(0,0)\foreach \y in {1,...,5}{\class(0,\y) \structline }}
\DeclareSseqGroup\towerfour{}{\class(0,0)\foreach \y in {1,...,4}{\class(0,\y) \structline }}
\DeclareSseqGroup\towerthree{}{\class(0,0)\foreach \y in {1,...,3}{\class(0,\y) \structline }}

\begin{sseqpage}[classes = fill, class labels = { left = 0.2em }, x tick step = 2, xscale = 1.5, yscale = 0.8, no ticks,class placement transform = {scale = 2}]	
	\class[white](0,0)\class[white](7,0)
	
	\class[black](2,0) \class[black](2,1) \class[black](2,2) \class[black](2,3) \class[black](2,4) \class[white](2,5)
	\draw[magenta](1.74,0)--(1.74,1); \draw[magenta](1.74,1)--(1.74,2); \draw[dashed, magenta](1.74,2)--(1.74,3); \draw[magenta](1.74,3)--(1.74,4);
	\draw[magenta](2.26,1)--(2.26,2); \draw[dashed, magenta](2.26,2)--(2.26,3);
	\draw[magenta](2.26,3)--(2.26,4); \draw[magenta](2.26,4)--(2.26,5);
	\draw[cyan](2,0)--(2.26,1);\draw[cyan](2,1)--(2.26,2);\draw[cyan](2,3)--(2.26,4);\draw[cyan](2,4)--(2.26,5);
	\draw[cyan](2,2)--(1.93,2.36);
	\class[white](2,0) \class[black](2,1) \class[black](2,2) \class[black](2,3) \class[black](2,4) \class[black](2,5)
	
	\class[black](4,0) \class[black](4,1) \class[black](4,2) \class[black](4,3) \class[black](4,4) \class[white](4,5)
	\draw[magenta](3.74,0)--(3.74,1); \draw[magenta](3.74,1)--(3.74,2); \draw[dashed, magenta](3.74,2)--(3.74,3); \draw[magenta](3.74,3)--(3.74,4);
	\draw[magenta](4.26,1)--(4.26,2); \draw[dashed, magenta](4.26,2)--(4.26,3);
	\draw[magenta](4.26,3)--(4.26,4); \draw[magenta](4.26,4)--(4.26,5);
	\draw[cyan](4,0)--(4.26,1);\draw[cyan](4,1)--(4.26,2);\draw[cyan](4,3)--(4.26,4);\draw[cyan](4,4)--(4.26,5);
	\draw[cyan](4,2)--(3.93,2.36);
	\class[white](4,0) \class[black](4,1) \class[black](4,2) \class[black](4,3) \class[black](4,4) \class[black](4,5)
	
	\class[black](6,0) \class[black](6,1) \class[black](6,2) \class[black](6,3) \class[black](6,4) \class[white](6,5)
	\draw[magenta](5.74,0)--(5.74,1); \draw[magenta](5.74,1)--(5.74,2); \draw[dashed, magenta](5.74,2)--(5.74,3); \draw[magenta](5.74,3)--(5.74,4);
	\draw[magenta](6.26,1)--(6.26,2); \draw[dashed, magenta](6.26,2)--(6.26,3);
	\draw[magenta](6.26,3)--(6.26,4); \draw[magenta](6.26,4)--(6.26,5);
	\draw[cyan](6,0)--(6.26,1);\draw[cyan](6,1)--(6.26,2);\draw[cyan](6,3)--(6.26,4);\draw[cyan](6,4)--(6.26,5);
	\draw[cyan](6,2)--(5.93,2.36);
	\class[white](6,0) \class[black](6,1) \class[black](6,2) \class[black](6,3) \class[black](6,4) \class[black](6,5)
	
	\node[background] at (0,-1) {\cdots};
	\node[background] at (2,-1) {k-5\geqslant 0}; \node[background] at (4,-1) {k}; \node[background] at (6,-1) {k+5};
	\node[background] at (-1,0) {0};\node[background] at (-1,1) {4};
	\node[background] at (-1,2) {8};
	\node[background] at (-1,2.6) {\vdots};\node[background] at (-1.3,3) {4m-4}; \node[background] at (-1.05,4) {4m};
	\node[background] at (-1.3,5) {4m+4};
	
	\draw[blue,->](2,1)--(4,0);\draw[blue,->](2,4)--(4,3);
	\draw[blue,->](2,2,2)--(4,1,2);\draw[blue,->](2,3,2)--(3.52,2.37);\draw[blue,->](2,5,2)--(4,4,2);
	
	\draw[blue,->](4,1)--(6,0);\draw[blue,->](4,4)--(6,3);
	\draw[blue,->](4,2,2)--(6,1,2);\draw[blue,->](4,3,2)--(5.52,2.37);\draw[blue,->](4,5,2)--(6,4,2);
	
	\foreach \foreach \y in {0,1,2,3,4,5}\draw[lightgray!50, thin](0,\y)--(8,\y);
	
	\node[background] at (1.33,0.05) {t^{k-5}};\node[background] at (1.25,1.05) {t^{k-5}a};\node[background] at (1.15,4.05) {t^{k-5}a^m};
	\node[background] at (3.48,0.05) {t^{k}};\node[background] at (3.40,1.05) {t^{k}a};\node[background] at (3.30,4.05) {t^{k}a^m};
	\node[background] at (5.33,0.05) {t^{k+5}};\node[background] at (5.25,1.05) {t^{k+5}a};\node[background] at (5.15,4.05) {t^{k+5}a^m};
	
	\node[background] at (2.75,1.05) {t^{k-5}b};\node[background] at (2.95,5.05) {t^{k-5}a^mb};
	\node[background] at (4.60,1.05) {t^{k}b};\node[background] at (4.65,2.05) {t^{k}ab};\node[background] at (4.8,5.05) {t^{k}a^mb};
	\node[background] at (6.75,1.05) {t^{k+5}b};\node[background] at (6.95,5.05) {t^{k+5}a^mb};	
\end{sseqpage}
\caption{$E_5$-term and $d_5$-differentials in \textbf{Case(\romannumeral 2)}}\label{case1-d5-2}
\end{figure}
Then $E_6^{k,l}$ is the same as in {\bf Case(\romannumeral 1)},
\begin{align*}
	E_6^{k,l} & =
	\begin{cases}
		\zz_2 , & 0\leqslant k\leqslant 4; l=0,4,8,\dots ,4m,\\
		0 , & \text{otherwise}.
	\end{cases}
\end{align*} 
Thus the cohomology groups $H^j(X_G)$ are also the same as in {\bf  Case(\romannumeral 1)}, 
\begin{align*}
	H^j(X_G) & =
	\begin{cases}
		\zz_2 , & 0\leqslant j\leqslant 4m+4 \text{~and~} j \ne 4,8,\dots ,4m,\\
		(\zz_2)^2 , & j=4,8,\dots ,4m,\\
		0 , & \text{otherwise}.
	\end{cases}
\end{align*} 

As $E_\infty ^{5,0}=0$, by \eqref{eq-P26}, we have $x^5=0$. Notice that, the elements $1\otimes a^2\in E_2^{0,8}$ and $1\otimes b\in E_2^{0,4}$ are permanent cocycles and are not hit by any $d_r$-coboundaries. Hence, they determine nontrivial elements $u\in E_\infty ^{0,8}$ and $v\in E_\infty ^{0,4}$, respectively. We have $u^{\frac{m+1}{2} }=0$ as $a^{m+1}=0$, and $v^2=0$ as $b^2=0$. Thus
\[
\text{Tot}E_\infty ^{*,*}\cong \zz_2[x,u,v]/\langle x^5,u^{\frac{m+1}{2}},v^2 \rangle,
\] 
where $\deg x=1$, $\deg u=8$, $\deg v=4$. 

Let $y\in H^8(X_G)$ and $z\in H^4(X_G)$ be such that $i^*(y)=a^2$ and $i^*(z)=b$, respectively. Similar to {\bf Case(\romannumeral 1)}, by \eqref{eq-hpm-case1Hj=totalCPX}, we get the following relations: 
\begin{align*}
	y^{\frac{m+1}{2}} & =\beta x^4y^{\frac{m-1}{2}}z,~\beta \in \zz_2,\\
	z^2 & =\alpha x^4z,~\alpha \in \zz_2.
\end{align*}
Therefore, 
\[
H^*(X_G)=\zz_2[x,y,z]/\langle x^5,y^{\frac{m+1}{2}}+\beta x^4y^{\frac {m-1}{2}}z, z^2+\alpha x^4z\rangle,
\] 
where $\deg x=1$, $\deg y=8$, $\deg z=4$, $\alpha,\beta \in \zz_2$ and $m$ is odd. If $m=1$, then $\beta =0$. 

By combining results in {\bf Case(\romannumeral 1)} and {\bf (\romannumeral 2)}, we can rewrite the result as follows
\[
H^*(X_G)=\zz_2[x,y,z]/\langle x^5,y^{\frac{m+1}{2}}+\beta x^4y^{\frac {m-1}{2}}z,z^2+\gamma y+\alpha x^4z\rangle,
\] 
where $\deg x=1$, $\deg y=8$, $\deg z=4$, $\alpha,\beta,\gamma\in \zz_2$ and $m$ is odd. If $m=1$, then $\beta=0, \gamma=0$. This gives possibility (1) of \Cref{thm-Z2faHPmS4}.

\textbf{Case(\romannumeral 3)}
$d_5(1\otimes a)=0$ and $d_5(1\otimes b)\ne 0$.

Immediately, $d_5(1\otimes b)=t^5\otimes 1$, so we have
\begin{align*}
	\begin{cases}
		d_5(1\otimes a^j)=0, & 1\leqslant j\leqslant m, \\
		d_5(1\otimes a^jb)=t^5\otimes a^j, & 0\leqslant j\leqslant m.
	\end{cases}
\end{align*}
and
\[
\xymatrix@R=2mm{
	E_5^{k-5,l+4}\ar[r]^-{d_5} & E_5^{k,l}\ar[r]^-{d_5} & E_5^{k+5,l-4},\\
	t^{k-5}\otimes a^{\frac{l}{4}}b\ar@{|->}[r]^-{d_5} & t^k\otimes a^{\frac{l}{4}}\ar@{|->}[r]^-{d_5} & 0,  \\
	t^{k-5}\otimes a^{\frac{l}{4}+1}\ar@{|->}[r]^-{d_5} & 0,~t^k\otimes a^{\frac{l}{4}-4}b\ar@{|->}[r]^-{d_5} & t^{k+5}\otimes a^{\frac{l}{4}-4}.}
\]
Then $E_6^{k,l}$ is the same as in {\bf Case(\romannumeral 1)},
\begin{align*}
	E_6^{k,l} & =
	\begin{cases}
		\zz_2 , & 0\leqslant k\leqslant 4; l=0,4,\dots ,4m,\\
		0 , & \text{otherwise}.
	\end{cases}
\end{align*} 
Thus the cohomology groups $H^j(X_G)$ are also the same as in {\bf  Case(\romannumeral 1)}, 
\begin{align*}
	H^j(X_G) & =
	\begin{cases}
		\zz_2 , & 0\leqslant j\leqslant 4m+4 \text{~and~} j \ne 4,8,\dots ,4m,\\
		(\zz_2)^2 , & j=4,8,\dots ,4m,\\
		0 , & \text{otherwise}.
	\end{cases}
\end{align*} 

As $E_\infty ^{5,0}=0$,  by \eqref{eq-P26}, we have $x^5=0$. Notice that, the element $1\otimes a\in E_2^{0,4}$ is a permanent cocycle and is not a $d_r$-coboundary. Hence, it determines a nontrivial element $u\in E_\infty ^{0,4}$. As we have remarked, $a^{m+1}=0$, so $u^{m+1}=0$. Thus
\[
\text{Tot}E_\infty ^{*,*}\cong \zz_2[x,u]/\langle x^5,u^{m+1}\rangle,
\] 
where $\deg x=1$, $\deg u=4$. 

Choose $y'\in H^4(X_G)$ such that $i^*(y')=a$ and let $y = y'+\alpha x^4\in H^4(X_G)$, $\alpha \in \zz_2$.
we get the following relation: 
\[
y^{m+1}=0.
\] 
Therefore, 
\[
H^*(X_G)=\zz_2[x,y]/\langle x^5,y^{m+1}\rangle,
\] 
where $\deg x=1$, $\deg y=4$. This gives possibility (2) of \Cref{thm-Z2faHPmS4}. \qed

\section{Applications to $\zz_2$-equivariant maps}\label{sec-equivmap}

We will now use the above results to study the existence of equivariant maps to and from $X$. This is an application that we find highly motivating. Let $X$ be a compact Hausdorff space with a free involution and the unit $n$-sphere $S^n$ carries the antipodal involution. Let us recall some numerical indices. 
\begin{definition}[{\cite{ConnerFloyd1960}}]\label{ConnerFloyddefind} 
The index of the involution on $X$ is 
\[
{\rm ind} (X)=\max\left\{ n\mid \text{there exists a } \zz_2 \text{-equivariant map } S^n\to X\right \}.  
\]
\end{definition}
\begin{definition}[{\cite{ConnerFloyd1960}}]\label{ConnerFloyddefco} 
The mod 2 cohomology index of the involution on $X$ is
\[
{\rm co}\text{-}{\rm ind}_2(X)=\max\left \{ n\mid \omega^n \ne 0\right \},  
\]
where $\omega \in H^1(X/\zz_2;\zz_2)$ is the Whitney class of the principal $\zz_2$-bundle $X\to X/\zz_2$.
\end{definition}
The above index and co-index are both defined by Conner and Floyd. Further, they gave the relationship between these indices. 
\begin{proposition}[{\cite{ConnerFloyd1960}}]\label{ConnerFloydindco} 
The following holds: {\rm ind}$(X)\leqslant$ {\rm co-ind}$_2(X).$
\end{proposition}
Given a $G$-space $X$, Volovikov defined a numerical index $i(X)$ as the following.
\begin{definition}[{\cite{Volovikov2000}}]\label{VolovikovdefiX}
The index $i(X)$ is the smallest $r$ such that for some $k,~d_r:E_r^{k-r,r-1} \to E_r^{k,0}$ in the cohomology Leray-Serre spectral sequence of the fibration $X\overset{i} \hookrightarrow X_G\overset{\pi}\to B_G$ is nontrivial.
\end{definition}

Let $\beta_k(X)$  be the $k$-th \emph{Betti number} of the space $X$.  Using Volovikov index, Coelho, Mattos and Santos proved the following results.
\begin{proposition}[{\cite[Theorem 1.1]{Coelho2012}}]\label{Coelhothm11} 
Let $G$ be compact Lie group and $X,Y$ be Hausdorff, path-connected and paracompact free $G$-spaces. With a PID as the coefficient for the cohomology, suppose that $i(X)\geqslant l+1$ for some natural $l\geqslant 1$ and $H^{k+1}(Y/G)=0$ for some $1 \leqslant k \leqslant l$.
\begin{enumerate}[{\rm (i)}]\itemindent=1em
\item If $k=l$ and $\beta _l(X)<\beta _{l+1}(B_G)$, then there is no $G$-equivariant map $f:X\to Y$.
\item If $1 \leqslant k<l$ and $0<\beta _{k+1}(B_G)$, then there is no $G$-equivariant map $f:X\to Y$.
\end{enumerate}
\end{proposition}

Using these Conner and Floyd indices, we get the following results.
\begin{proposition}\label{prop-SnX}  
Let $X\sim _2 \mathbb{R}P^m \times S^4$ be a finitistic space with a free involution and consider the antipodal involution on $S^n$. If $m=5$ or $m=7$, assume further that the action of $\zz_2$ on $H^*(X;\zz_2)$ is trivial or $X\sim_{\zz} \mathbb{R}P^m \times S^4$. Then the mod 2 co-index of $X$ can only take the values $C=1,4,m+2,m+3$ and $m+4$ and there are no $\zz_2$-equivariant maps $S^n\to X$ for $n \geqslant C+1$.
\end{proposition}  
\begin{proof}
For the principal $\zz_2$-bundle $X\to X/\zz_2$, we can take a classifying map
\[
f:X/\mathbb{Z}_2 \to B_{\mathbb{Z}_2}.
\]
It would uniquely determine a homotopy class of $[X/\zz_2,B_{\zz_2}]$. Let $\eta:X/\zz_2\to X_{\zz_2}$ is a homotopy inverse of the homotopy equivalence $h:X_{\zz_2}\to X/\zz_2$, then $\pi\eta:X/\zz_2\to B_{\zz_2}$ also classifies the principal $\zz_2$-bundle $X\to X/\zz_2$. Therefore, we find the following homotopy equivalence $f \simeq \pi\eta$. Consider the map
\[
\pi^*:H^1(B_{\mathbb{Z}_2})\to H^1(X_{\mathbb{Z}_2}).
\]
The characteristic class $t\in H^1(B_{\zz_2})$ of the universal bundle $\zz_2\hookrightarrow E_{\zz_2}\overset{\pi}\to B_{\zz_2}$ is mapped to $\pi^*(t)\in H^1(X_{\zz_2})\cong H^1(X/\zz_2)$, which is the Whitney class of the principal $\zz_2$-bundle $X\to X/\zz_2$. 

For $X\sim _2 \mathbb{R}P^m \times S^4$, by possibility (1) of \Cref{thm-Z2faRPmS4}, we see that $x\ne 0$ and $x^2=0$. Thus, co-ind$_2(X)=1$. By \Cref{ConnerFloydindco}, ind$(X)\leqslant 1$, this means that there is no $\zz_2$-equivariant map $S^n\to X$ for $n\geqslant 2$. 

In possibility (2) of \Cref{thm-Z2faRPmS4}, $x^4\ne 0$ and $x^5=0$. Accordingly, {\rm co-ind}$_2(X)=4$, ind$(X)\leqslant 4$ and there is no $\zz_2$-equivariant map $S^n\to X$ for $n\geqslant 5$. 

In possibilities (3), (4), (6) and (8) of \Cref{thm-Z2faRPmS4}, $x^{m+4}\ne 0$ and $x^{m+5}=0$. Accordingly, {\rm co-ind}$_2(X)=m+4$, ind$(X)\leqslant m+4$ and there is no $\zz_2$-equivariant map $S^n\to X$ for $n \geqslant m+5$. 

In possibilities (5) and (7) of \Cref{thm-Z2faRPmS4}, $x^{m+3}\ne 0$ and $x^{m+4}=0$. Therefore, {\rm co-ind}$_2(X)=m+3$, ind$(X)\leqslant m+3$ and there is no $\zz_2$-equivariant map $S^n\to X$ for $n \geqslant m+4$. 

Finally, in possibility (9) of \Cref{thm-Z2faRPmS4}, $x^{m+2}\ne 0$ and $x^{m+3}=0$. Thus, we have {\rm co-ind}$_2(X)=m+2$, ind$(X)\leqslant m+2$ and there is no $\zz_2$-equivariant map $S^n\to X$ for $n \geqslant m+3$. 
\end{proof}

By a similar proof, we get the following results for the $\zz_2$-equivariant maps from $S^n$ to $X\sim _2 \mathbb{C}P^m \times S^4$ or $X\sim_2\mathbb{H}P^m \times S^4$.
\begin{proposition}\label{prop-Sn-CPmS4} 
Let $X\sim _2 \mathbb{C}P^m \times S^4$ be a finitistic space with a free involution and consider the antipodal involution on $S^n$. If $m=3$, assume further that the action of $\zz_2$ on $H^*(X;\zz_2)$ is trivial or $X\sim _{\zz} \mathbb{C}P^3 \times S^4$. Then the mod 2 co-index of $X$ can only take the values $C=2,4$ and $2m+4$ and there are no $\zz_2$-equivariant maps $S^n\to X$ for $n \geqslant C+1$. 
\end{proposition} 
\begin{proposition}\label{prop-Sn-HPmS4} 
Let $X\sim _2 \mathbb{H}P^m \times S^4$ be a finitistic space with a free involution and consider the antipodal involution on $S^n$. When $m\equiv 3\pmod 4$, assume further that the action of $\zz_2$ on $H^*(X;\zz_2)$ is trivial or $X\sim _{\zz} \mathbb{H}P^m \times S^4$. Then the mod 2 co-index of $X$ can only take the value $4$ and there are no $\zz_2$-equivariant maps $S^n\to X$ for $n \geqslant 5$. 
\end{proposition} 

Note that, the index of $X\sim _2 \mathbb{F}P^m \times S^4$ (\Cref{ConnerFloyddefind}) can be no more than $m+4$, $2m+4$ and $4$, when $\ff=\rr$, $\cc$ or $\hh$ respectively. 

We get the following immediate consequences by the proof of \Cref{thm-Z2faRPmS4}, \Cref{thm-Z2faCPmS4} and \Cref{thm-Z2faHPmS4}.
\begin{proposition}\label{prop-iX} 
Let $\mathbb{Z}_2$ act freely on a finitistic space $X\sim _2 \mathbb{R}P^m \times S^4$. If $m=5$ or $m=7$, assume further that the action of $\zz_2$ on $H^*(X;\zz_2)$ is trivial or $X\sim_{\zz} \mathbb{R}P^m \times S^4$.  Then $i(X)$ has one of the following values: $2,~5,~m+3,~m+4~or~m+5$.
\end{proposition}	
\begin{proposition}\label{prop-iX-CPmS4} 
Let $\mathbb{Z}_2$ act freely on a finitistic space $X\sim _2 \mathbb{C}P^m \times S^4$. If $m=3$, assume further that the action of $\zz_2$ on $H^*(X;\zz_2)$ is trivial or $X\sim _{\zz} \mathbb{C}P^3 \times S^4$. Then $i(X)$ has one of the following values: $3,~5~or~2m+5$.
\end{proposition}
\begin{proposition}\label{prop-iX-HPmS4} 
Let $\mathbb{Z}_2$ act freely on a finitistic space $X\sim _2 \mathbb{H}P^m \times S^4$. When $m\equiv 3\pmod 4$, assume further that the action of $\zz_2$ on $H^*(X;\zz_2)$ is trivial or $X\sim _{\zz} \mathbb{H}P^m \times S^4$. Then $i(X)=5$.
\end{proposition}

By \Cref{Coelhothm11} and \Cref{prop-iX}, we obtain
\begin{proposition}\label{prop-XY} 
Suppose that $\mathbb{Z}_2$ acts freely on a finitistic space $X\sim _2 \mathbb{R}P^m \times S^4$ and path-connected, paracompact Hausdorff space $Y$. If $m=5$ or $m=7$, assume further that the action of $\zz_2$ on $H^*(X;\zz_2)$ is trivial or $X\sim_{\zz} \mathbb{R}P^m \times S^4$. Then there is no $\zz_2$-equivariant map $X\to Y$
\begin{enumerate}[{\rm (a)}]\itemindent=2em
	\item If $i(X)=5$ and $H^k(Y/\zz_2)=0$ for some $2\leqslant k<5$;
	\item If $i(X)=m+3$ and $H^k(Y/\zz_2)=0$ for some $2\leqslant k<m+3$;
	\item If $i(X)=m+4$ and $H^k(Y/\zz_2)=0$ for some $2\leqslant k<m+4$;
	\item If $i(X)=m+5$ and $H^k(Y/\zz_2)=0$ for some $2\leqslant k<m+5$.
\end{enumerate}
\end{proposition}	
\begin{proof}
We observe that $\beta _l(B_{\zz_2};\zz_2)=1$ for all $l$. By \Cref{prop-iX},  $i(X)$ is one of 2, 5, $m+3$, $m+4$ or $m+5$. We can apply these results to \Cref{Coelhothm11}. If $i(X)=5, m+3, m+4\text{~or~}m+5$, then we get the possibilities (a), (b), (c) or (d), respectively. 
\end{proof}

For the same reason, we obtain the following propositions directly.
\begin{proposition}\label{prop-XY-CPmSn} 
Suppose that $\mathbb{Z}_2$ acts freely on a finitistic space $X\sim _2 \mathbb{C}P^m \times S^4$ and path-connected, paracompact Hausdorff space $Y$. If $m=3$, assume further that the action of $\zz_2$ on $H^*(X;\zz_2)$ is trivial or $X\sim _{\zz} \mathbb{C}P^3 \times S^4$. Then there is no $\zz_2$-equivariant map $X\to Y$
\begin{enumerate}[{\rm (a)}]\itemindent=2em
	\item If $i(X)=3$ and $H^k(Y/\zz_2)=0$ for $k=2$;
	\item If $i(X)=5$ and $H^k(Y/\zz_2)=0$ for some $2\leqslant k<5$;
	\item If $i(X)=2m+5$ and $H^k(Y/\zz_2)=0$ for some $2\leqslant k<2m+5$.
\end{enumerate}
\end{proposition}
\begin{proposition}\label{prop-XY-HPmS4} 
Suppose that $\mathbb{Z}_2$ acts freely on a finitistic space $X\sim _2 \mathbb{H}P^m \times S^4$ and path-connected, paracompact Hausdorff space $Y$. When $m\equiv 3\pmod 4$, assume further that the action of $\zz_2$ on $H^*(X;\zz_2)$ is trivial or $X\sim _{\zz} \mathbb{H}P^m \times S^4$. If $i(X)=5$ and $H^k(Y/\zz_2)=0$ for some $2\leqslant k<5$, then there is no $\zz_2$-equivariant map $X\to Y$.
\end{proposition}

Replacing $Y$ in the above by $S^n$, we obtain the following results.
\begin{corollary}\label{coro-XSn} 
Let $X\sim _2 \mathbb{R}P^m \times S^4$ be a finitistic space and the unit $n$-sphere $S^n$ be equipped with a free involution. If $m=5$ or $m=7$, assume further that the action of $\zz_2$ on $H^*(X;\zz_2)$ is trivial or $X\sim_{\zz} \mathbb{R}P^m \times S^4$. Then, there is no $\zz_2$-equivariant map $X\to S^n$
\begin{enumerate}[{\rm (a)}]\itemindent=2em
	\item If $i(X)=5$ and $n<4$;
	\item If $i(X)=m+3$ and $n<m+2$;
	\item If $i(X)=m+4$ and $n<m+3$;
	\item If $i(X)=m+5$ and $n<m+4$.
\end{enumerate}
\end{corollary}
\begin{corollary}\label{coro-XSn-CPmS4} 
Let $X\sim _2 \mathbb{C}P^m \times S^4$ be a finitistic space and the unit $n$-sphere $S^n$ be equipped with a free involution. If $m=3$, assume further that the action of $\zz_2$ on $H^*(X;\zz_2)$ is trivial or $X\sim _{\zz} \mathbb{C}P^3 \times S^4$. Then, there is no $\zz_2$-equivariant map $X\to S^n$
\begin{enumerate}[{\rm (a)}]\itemindent=2em
	\item If $i(X)=3$ and $n<2$;
	\item If $i(X)=5$ and $n<4$;
	\item If $i(X)=2m+5$ and $n<2m+4$.
\end{enumerate}
\end{corollary}
\begin{corollary}\label{coro-XSn-HPmS4} 
Let $X\sim _2 \mathbb{H}P^m \times S^4$ be a finitistic space and the unit $n$-sphere $S^n$ be equipped with a free involution. When $m\equiv 3\pmod 4$, assume further that the action of $\zz_2$ on $H^*(X;\zz_2)$ is trivial or $X\sim _{\zz} \mathbb{H}P^m \times S^4$. If $i(X)=5$ and $n<4$, then there is no $\zz_2$-equivariant map $X\to S^n$.
\end{corollary}

\paragraph{\textbf{Acknowledgments}}  The authors would like to thank the anonymous referee who provided many constructive, useful and detailed comments on the original manuscript, which led to significant improvements in  the presentation of this paper.

\end{document}